\documentclass{amsart}
\usepackage[utf8]{inputenc}
\usepackage{amsfonts}
\usepackage[T1]{fontenc}
\usepackage{hyperref}
\hypersetup{
	pdftitle={},
	pdfsubject={},
	pdfauthor={},
	pdfkeywords={}
}
\usepackage{amsmath,amssymb}
\usepackage{xcolor}
\usepackage{amsthm}
\usepackage{pdflscape}
\usepackage{float}
\usepackage{graphicx}
\usepackage{mathtools}
\usepackage{todonotes}
\usepackage{autonum}
\usepackage{pdflscape}
\usepackage{longtable}
\usepackage{breqn}
\usepackage{textcomp}
\usepackage{listings}
\usepackage[foot]{amsaddr}
\usepackage{comment}

\newcommand{\Real}{\mathbb{R}} 
\newcommand{\Complex}{\mathbb{C}}
\newcommand{\Integer}{\mathbb{Z}} 

\newcommand\bydef{\stackrel{\mathclap{\normalfont\mbox{\tiny def}}}{=}}

\newtheorem{theorem}{Theorem}[section]
\newtheorem{corollary}[theorem]{Corollary}
\newtheorem{proposition}[theorem]{Proposition}
\newtheorem{lemma}[theorem]{Lemma}
\newtheorem{conjecture}[theorem]{Conjecture}
\newtheorem{prop/def}[theorem]{Definition/Proposition}

\newtheorem{question}[theorem]{Open question}

\theoremstyle{definition}
\newtheorem{definition}[theorem]{Definition}

\newtheorem{example}[theorem]{Example}

\theoremstyle{remark}
\newtheorem{remark}[theorem]{Remark}

\numberwithin{theorem}{section}
\numberwithin{equation}{section}

\title{Origami: real structure, enumeration and quantum modularity}

\author[R.~Fesler]{Raphaël Fesler}
\address{R.~Fesler: Guangdong Technion-Israel Institute of Technology, Daxue lu 241, 515063, Shantou, Guangdong,
P.R. of China}
\email{raphael.fesler@gtiit.edu.cn}
\author[P.~Zograf]{Peter Zograf}
\address{P.~Zograf: St. Petersburg Department of Steklov Mathematical Institute, St.~Petersburg 191023, Russia
}
\email{zograf@pdmi.ras.ru}

\thanks{This work is supported by the Russian Science Foundation  (project 25-11-00058)}

\begin{document}
	
	\begin{abstract}
	We define real origami (that is, origami equipped with a real structure) and enumerate them using the combinatorics of zonal polynomials. We explicitly express in terms of sums of divisors the numbers of genus 2 real origami with 2 simple zeros and the numbers of genus 3 real origami with 2 double zeros showing that their 
    generating functions are quantum modular forms. Furthermore, 
    we show that by replacing zonal polynomials with Schur polynomials we can effectively count the classical (complex) origami. 
    As a by-product, we establish a connection between classical origami and a specific class of double Hurwitz numbers.
    Finally, we discuss some conjectures and open questions involving Jack functions, quantum modular forms, and integrable hierarchies.
        
	\end{abstract}
	
	\maketitle
	\tableofcontents
	
	\section{Introduction}
	
	An {\it origami} $X$ is a square-tiled surface of a special kind that can be constructed as follows. Take $n$ identical unit squares $S=[0,1] \times [0,1]$ and label their opposite sides as "top"--"bottom" and "left"--"right". Glue the squares together by identifying the top and bottom sides and the left and right sides, reversing the orientation. As a result, we get a compact surface without boundary that carries a natural complex structure and a holomorphic Abelian differential. We will call such an object a {\it complex} origami to distinguish it from a
    {\it real} origami that we will define a little later.
	
	Recall that a complex origami is conveniently represented by two permutations $h,v$ in the symmetric group $S_n$ (where $h$ stands for "horizontal" and $v$ stands for "vertical"). Moreover, each origami $X$ is a ramified cover of the torus $T=\mathbb{C}/\mathbb{Z}\oplus\!\sqrt{-1}\,\mathbb{Z}$, where the projection $f:X\to T$ is holomorphic with only one critical value at $0\in T$. The ramification profile of $f$ over $0\in T$ is described by the cycle structure of the commutator $[h,v]=hvh^{-1}v^{-1}\in S_n$. Clearly, $X$ is connected if and only if $h$ and $v$ generate a subgroup in $S_n$ that acts transitively on the set $\{1,\ldots,n\}$. 
    (See \cite{Ker09,Zorich2002} and the references therein for an accessible introduction.)
	
	In the present paper we will predominantly deal with origami admitting an 
    anti-holomorphic involution. We consider two kinds of such origami.
    The first one, that we call {\it real origami}, carries a fixed point free anti-holomorphic
    involution that covers the ordinary complex conjugation $z\mapsto \bar{z}$ on the torus $T$. The second one, that we call {\it mirror symmetric}, or simply {\it mirror origami}, carries a fixed point free anti-holomorphic involution that covers the reflection 
    $(x,y)\mapsto (y,x),\;z=x+\sqrt{-1}y\in T$ (precise definitions are below).
    As we will see, the real and mirror origami, though different, are in bijection.
    
    Effective origami enumeration is a long-standing combinatorial problem. The main objective of the present paper is to propose an independent 
    approach to real origami count based on the combinatorics of the symmetric group (more precisely, zonal polynomials), see Section \ref{Section:Zonal}.
    In the special cases of genus 2 and 2 simple zeros, as well as genus 3 and 2 
    double zeros, we explicitly express the number of real origami of a given degree
    (= the number of squares) in terms of divisor sums, or Eisenstein series 
    (Theorem \ref{Th:Genfunction} and Section \ref{g3}). The algorithm of the real origami count also applies to the case of complex origami (with zonal polynomials replaced by Schur polynomials), see Section \ref{Section:ComplexOrigami}. Finally, Section \ref{Section:JackFunction} discusses further connections of real origami with Jack polynomials, quantum modular forms, and integrable hierarchies. Some numerical data
    is given in Appendices and Tables at the end of the paper.

    \section{Real origami and zonal polynomials}\label{Section:Zonal}
    \subsection{Definitions and preliminaries}    
    We begin with two definitions of real origami and show their equivalence.
	
	\begin{definition}[geometric]\label{Def:OrigamiGeo}
		An origami $f:X\to T$ is called {\it real} if the following conditions are satisfied: 
		\begin{itemize}
			\item $f:X\to T$ is a ramified cover of $T$ branching over $0\in T$; 
			\item there exist an anti-holomorphic fixed point free involution $\phi:X\to X$ 
			such that $f\circ\phi={\rm conj}\circ f$, where ${\rm conj}:T\to T$ is the complex conjugation,
            and the factor space $X/\phi$ is connected;
			\item the preimage of the meridian $\beta=\{x=1/2\}\subset T$ splits into disjoint union of pairs of 
			non-intersecting cycles $f^{-1}(\beta)=(\sigma_1,\bar\sigma_1,\ldots,\sigma_m,\bar\sigma_m)$, where $\bar\sigma_i=\phi(\sigma_i^{-1}),\;i=1,\ldots,m$.
		\end{itemize}
	\end{definition}
	
	\begin{definition}
		[combinatorial]\label{Def:OrigamiCombi}
		A pair of permutations $h,v\in S_{2n}$ defines a real origami if the following conditions are
		satisfied:
		\begin{itemize} 
			\item there exist an involution without fixed points $\tau\in S_{2n}$ such that 
			$\tau h\tau h^{-1}=\tau v\tau v=id$;
                \item the cycle decomposition of $v$ consist of even number of cycles 
			$\gamma_1\bar\gamma_1\ldots\gamma_{m}\bar\gamma_m$, where $\bar\gamma_i = 
			\tau\gamma_i^{-1}\tau,\;i=1,\ldots,m$.
                \item the group generated by $h, v$ and $\tau$ acts transitively on the set
			$\{1,\bar{1},\ldots,n,\bar{n}\}$;
		\end{itemize}
	\end{definition}
    
	\begin{remark}
    The numbers that count real origami provide a version of twisted elliptic Hurwitz numbers, 
    cf. \cite{HM24}.
	   The combinatorial model described by Definition \ref{Def:OrigamiCombi} was introduced in \cite{HM24}
       in a different context of tropical geometry. The next proposition answers, in particular, a question from \cite{HM24}
       about a geometric counterpart of Definition \ref{Def:OrigamiCombi}.
        \end{remark}
    
	\begin{proposition}\label{Prop:DefEqui}
		Definitions \ref{Def:OrigamiGeo} and \ref{Def:OrigamiCombi} are equivalent.
		\begin{proof}
			Let us show that Definition \ref{Def:OrigamiGeo} implies Definition \ref{Def:OrigamiCombi}.
			Take $b=(1/2,1/2)\in T$ as the base point, and denote by $\alpha$ and $\beta$ the loops 
			that freely generate the fundamental group $\pi_1(T\setminus\{0\},b)$, see Fig.  \ref{Fig.Loops}.
			\begin{figure}[H]
				\includegraphics[scale=.6]{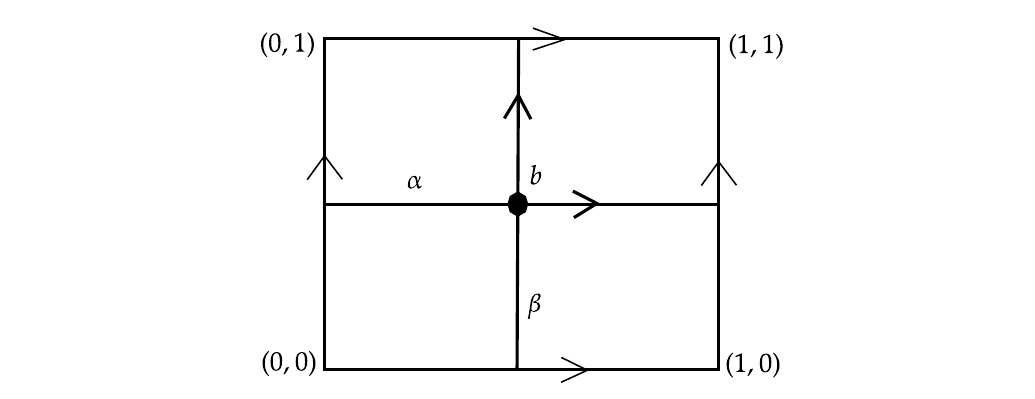}  
				\caption{Loops on the torus.}\label{Fig.Loops}
			\end{figure} 
			
			Consider the monodromy representation
			\[\rho:\pi_1(T\setminus\{0\},b)\to S_{2n},\]
			and put $h=\rho(\alpha),\;v=\rho(\beta)$. Notice that the preimage $f^{-1}(b)\in X$ consists of $2n$ distinct points 
			$a_1,\ldots,a_{2n}$ 
			that are pairwise interchanged by $\phi$, thus giving
			an involution $\tau$ in $S_{2n}$ without fixed points. It is an easy check that the triple
			$(h,v,\tau)$ satisfies all the conditions of Definition \ref{Def:OrigamiCombi}.
			
			Now let us show that Definition \ref{Def:OrigamiCombi} implies Definition \ref{Def:OrigamiGeo}. There are two cases to consider.

            \emph{Case 1}: The group generated by $h,\,v$ acts transitively on the set $\{1,\bar{1},\ldots,n,\bar{n}\}$. Then, using $h$ and $v$ we construct a connected ramified cover $f:X\to T$ with the prescribed ramification 
			profile over $0\in T$. To define the anti-holomorphic involution $\phi:X\to X$, take a regular point $x\in X$ and choose a path $s$ connecting $f(x)$ 
			with the base point $b=(1/2,1/2)\in T$. Lift the path $s$ to $X$ starting at $x$; it will end
			at some point $a_i\in f^{-1}(b)$. Consider the path ${\rm conj}\circ s^{-1}$ that is reverse and complex conjugate to $s$, and lift it to $X$ starting at $\tau a_i\in f^{-1}(b)$. It will end at some point in $f^{-1}(f(x))$ that we denote by $\phi(x)$. We need to verify that $\phi(x)$ is
			independent of the path $s$. This is easy to see when $f(x)$ belongs to the interior of the unit square $S$, in which case a path connecting $f(x)$ with $b$ inside $S$ is unique up to 
			homotopy. Let us show that $\phi(x)$ extends smoothly to the entire surface $X$; to do that, it is
			sufficient to consider $s=\alpha$ and $s=\beta$. In the first case lift the path $\alpha$ to $X$
			starting at some $a_i$; it will end at $h(a_i)\in f^{-1}(b)$. With $\tau h(a_i)$ as a start point, lift the reverse and complex conjugate of $\alpha$ to $X$; it will end at 
			$h^{-1}\tau h(a_i)=\tau a_i$, i.e. $\phi(a_i)=\tau a_i$ (here we used that ${\rm conj}(\alpha)=\alpha$ and $h^{-1}\tau h(a_i)=\tau$). In the second case, a similar argument 
			also yields $\phi(a_i)=\tau a_i$, thus proving that the map $\phi:X\to X$ is well defined.
			By construction, $\phi$ is anti-holomorphic and satisfies the relation 
			$f\circ\phi={\rm conj}\circ f$. Since $\tau$ is an involution, $\phi\circ\phi=\rm{id}$.
			Moreover, the condition that $\tau$ has no fixed points implies $\phi(x)\neq x$ for all $x\in X$.
			It is straightforward to verify the last condition in Definition \ref{Def:OrigamiGeo}
			provided it holds for Definition \ref{Def:OrigamiCombi}.

            {\emph Case 2}: The group generated by $h,\,v$ does not act transitively on the set $\{1,\bar{1},\ldots,n,\bar{n}\}$. In this case the covering surface $X$ has two connected components
            $X=Y\cup\overline{Y}$. Considerations similar to those of Case 1 show that these components are identified by an anti-holomorphic involution $\phi:Y\to \overline{Y}$, where $\overline{Y}$ is the complex conjugate of $Y$.
		\end{proof}
	\end{proposition}
	
	\begin{remark}\label{Rem:Profil}
		Since $\phi:X\to X$ is a fixed point free involution, the ramification profile of $f:X\to T$ has the form
		$\lambda=[\lambda_1^2\ldots\lambda_l^2]$ (i.e., each element of the partition $\lambda$ enters with even multiplicity). Equivalently, the commutator $[h,v]=hvh^{-1}v^{-1}$ has the same cycle structure
		$\lambda=[\lambda_1^2\ldots\lambda_l^2]$.
	\end{remark}
	
	\begin{remark}
		Denote by $\mathcal{O}_\Real^{geom}(2n)$ (resp. $\mathcal{O}_\Real^{comb}(2n)$) the set of real origami
		of degree $2n$ given by
		Definition  \ref{Def:OrigamiGeo} (resp. by Definition  \ref{Def:OrigamiCombi}). It is easy to see that the forgetful map $\mathcal{O}_\Real^{comb}(2n)\to\mathcal{O}_\Real^{geom}(2n)$ is
		independent of the labeling of the points in $f^{-1}(b)$ and therefore factors through $\pi:\mathcal{O}_\Real^{comb}(2n)/S_{2n}\to\mathcal{O}_\Real^{geom}(2n)$. It is standard to show that $\pi$ is a bijection (see \cite{CM16}). Recall that an automorphism of a cover $f$ can be identified with a permutation on $f^{-1}(b)$ yielding
		$$\text{Stab}_{S_{2n}}(h,v) = \text{Aut}_\Real (\pi(h,v)).$$
		This means that ${|\mathcal{O}_\Real^{comb}(2n)|}/{{(2n)!}}$ is the weighted number of real 
		origami (where each origami is counted with weight reciprocal to the order of its group
		of real automorphisms). 
	\end{remark}
	
	\subsection{Enumeration of real origami of genus 2}
	Here we derive an explicit formula for the number of real origami of arbitrary
	even degree $2n$ (equal to the number of squares) with 2 simple zeros following an approach of Zorich, 
	see \cite{Zorich2002}. Denote by $\mathcal{H}(\mu_1-1,\ldots,\mu_k-1)$ the moduli space of holomorphic Abelian 
	differentials with $k$ zeros of orders $\mu_1-1,\ldots,\mu_k-1$, where $\mu=[\mu_1\ldots\mu_k]$
	is a partition of $2n$. By the Riemann--Hurwitz formula, the genus of the underlying surface is then
	\[g=1+n-\frac{k}{2}\]
	(in particular, $k$ must be even).
	
    Let us proceed with simple examples. The case of Abelian covers (that is, $g=1$, or, equivalently,  $\mu=[1^{2n}]$) is well-known, see \cite{Zorich2002}. The next case $g=2$ is less trivial. The moduli space of genus 2 
	Abelian differentials consists of 2 strata --- the principal stratum $\mathcal{H}(1,1)$ and the 
	codimension 1 stratum $\mathcal{H}(2)$. Clearly, no real origami can belong to the stratum 
	$\mathcal{H}(2)$. It means that all genus 2 real origami lie in the principal stratum $\mathcal{H}(1,1)$. 
	We have	
	\begin{theorem}\label{Th:Genfunction}
		The number $N_n^\Real1 (1,1)=|\mathcal{O}_\Real^{geom}(2n)\cap\mathcal{H}(1,1)|$
        of genus 2 degree $2n$ real origami in $\mathcal{H}(1,1)$ is given by the formula
		\begin{equation}
			N_n^\Real (1,1)=\frac{1}{2}(\sigma_2(n)-\sigma_1(n))\;,
		\end{equation}
		where 
		\[\sigma_k(n) := \sum_{d|n} d^k\,\]
		and $d$ runs over the divisors of $n$. Equivalently, at the level of generating functions
\[ 
\sum_{n=1}^\infty N_n^\Real (1,1)q^n=\frac{1}{2}(E_3-E_2),         
\]
where 
\begin{align}\label{Eis}
E_k=\sum_{n=1}^\infty \sigma_{k-1}(n)q^n
\end{align}
is the $k$-th Eisenstein series, $k=2,3,\ldots$.
\end{theorem}
	\begin{proof}
		After A. Zorich \cite{Zorich2002}, cf. also \cite{DGM2021,GM2020}, we use separatrix diagrams and cylinder decompositions to count 
		real origami in strata of Abelian differentials. Roughly speaking, a separatrix diagram is the closure
		of the union of non-closed horizontal (or vertical) trajectories of a holomorphic Abelian differential. The 
		complement to a separatrix diagram is a disjoint union of flat cylinders, each foliated into circles.
		In the case of $\mathcal{H}(1,1)$, there are 4 admissible separatrix diagrams (i.e., realizable by a complex origami); see Figure 
        \ref{Fig.Separatrix1} below (details can be found in \cite{GM2020,Zorich2023}).
        
       \begin{figure}[h]
\begin{center}
\includegraphics[width=4in]{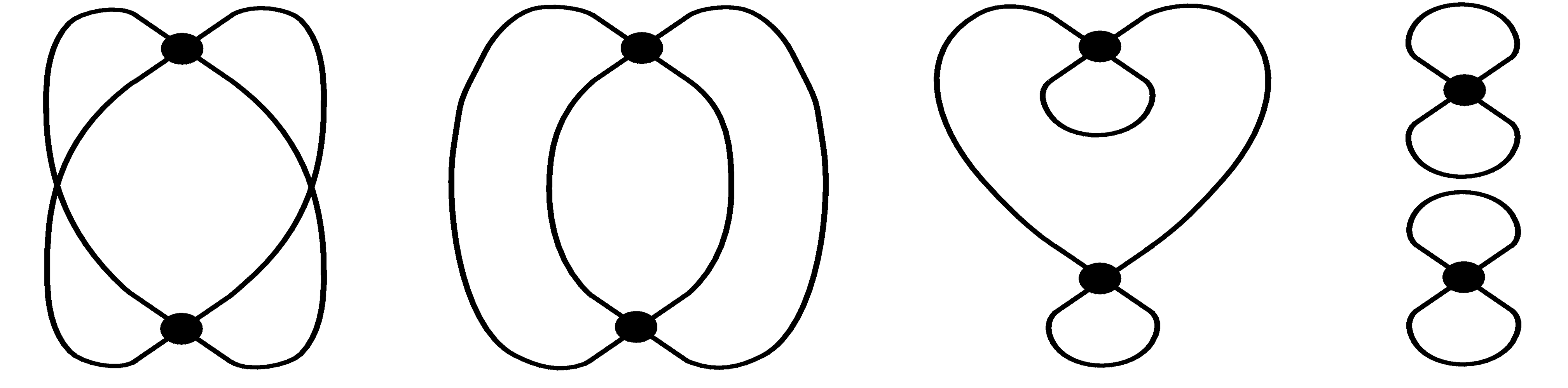}
\caption{From left to right: separatrix diagrams with one cylinder (type I), with
two cylinders (type IIa and type IIb), and with three cylinders (type III).}
\end{center}
        \label{Fig.Separatrix1}
	\end{figure}
 
        Out of these four diagrams only type IIa is compatible with a real structure. Indeed, types I and III are not realizable because they have an odd number of vertical cylinders, cf. Definition \ref{Def:OrigamiGeo}, and type IIb is not compatible with assignment of positive lengths to saddle connections (singular trajectories).
        
        \begin{figure}[h]
\begin{center}
\includegraphics[width=4in]{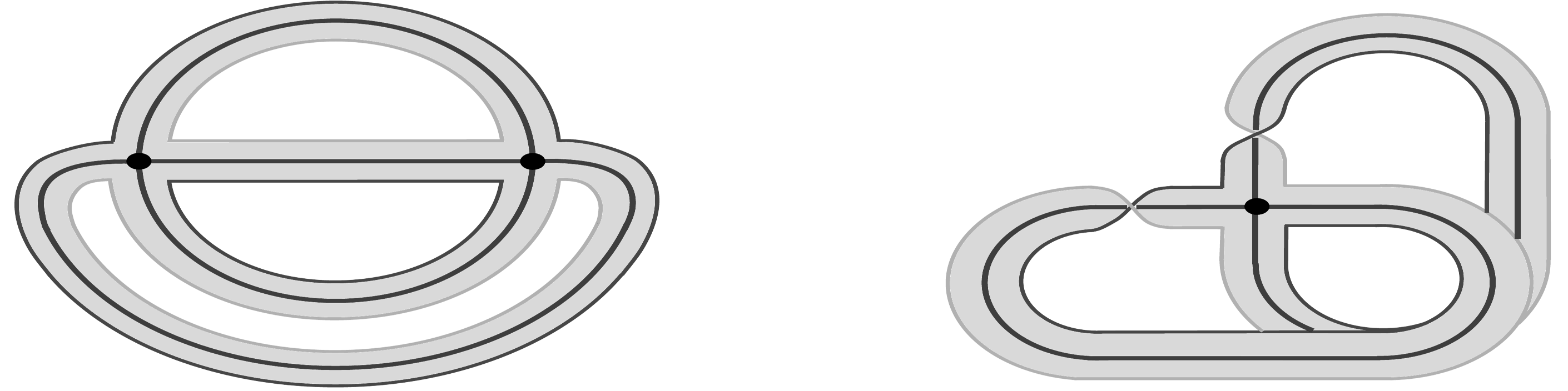}
\caption{2-cylinder separatrix diagram invariant with respect to the anti-holomorphic involution
(left) and its 1-cylinder non-orientable factor (right).}
\end{center}
        \label{Fig.Separatrix}
	\end{figure}
		
		Let us factorize separatrix diagram IIa by the anti-holomorphic involution. We  get a M\"obius
		graph with one vertex, two twisted loops and two boundary cycles identified with two boundary
		components of the cylinder. This gives us a non-orientable square-tiled surface composed of $n$ unit
		squares. Let $\ell_1,\ell_2$ be the lengths of the loops, then $\ell=\ell_1+\ell_2$ is the length 
		of each boundary component. We compute the number of ways to glue the cylinder of circumference $\ell$
		to the M\"obius graph. First, we notice that $\ell$ can be split into the sum $\ell_1+\ell_2$ 
		of positive integers in $\ell-1$ ways. Second, one boundary component of the cylinder can be shifted by
		$0,1,\ldots,\ell-1$ units relative to the other one. We also see that $n$ must be divisible by $\ell$.
All together, this yields
		\[
		N_n^\Real (1,1)=\frac{1}{2}\sum_{\ell|n} \ell(\ell-1),
		\]
		where the factor 1/2 accounts for the rotational symmetry of the M\"obius graph.
	\end{proof}
	
	\begin{corollary}
		The number of real origami in $\mathcal{H}(1,1)$ has the asymptotics
        \[
		N_n^\Real(1,1)=\frac{\zeta(3)}{6}n^3+O(n^2)
		\quad {\rm as}\; n\rightarrow \infty.
        \]
	\end{corollary}
	
	\begin{proof}
		It is well known that $\sigma_1(n)=\frac{\pi^2}{12}n^2 + O(n\log(n))$ and 
		$\sigma_2(n)=\frac{\zeta(3)}{3}n^3 + O(n^2)$ as $n\to\infty$,
		see, e.g., \cite{HW79}. 
	\end{proof}

    \subsection{Enumeration of real origami of genus 3}\label{g3} Here we sketch a proof of a formula for the numbers $N_n^\Real(2,2)$ that count real origami of degree $2n$ in the stratum $\mathcal{H}(2,2)$ of genus 3 Abelian differentials with 2 double poles.
    In this case there are 3 separatrix diagrams (out of 24 admissible ones) that make a nontrivial contribution
    to the numbers $N_n^\Real(2,2)$, see Figure \ref{Fig.Separatrix2}.  
\begin{figure}[h]
\begin{center}
\includegraphics[width=4in]{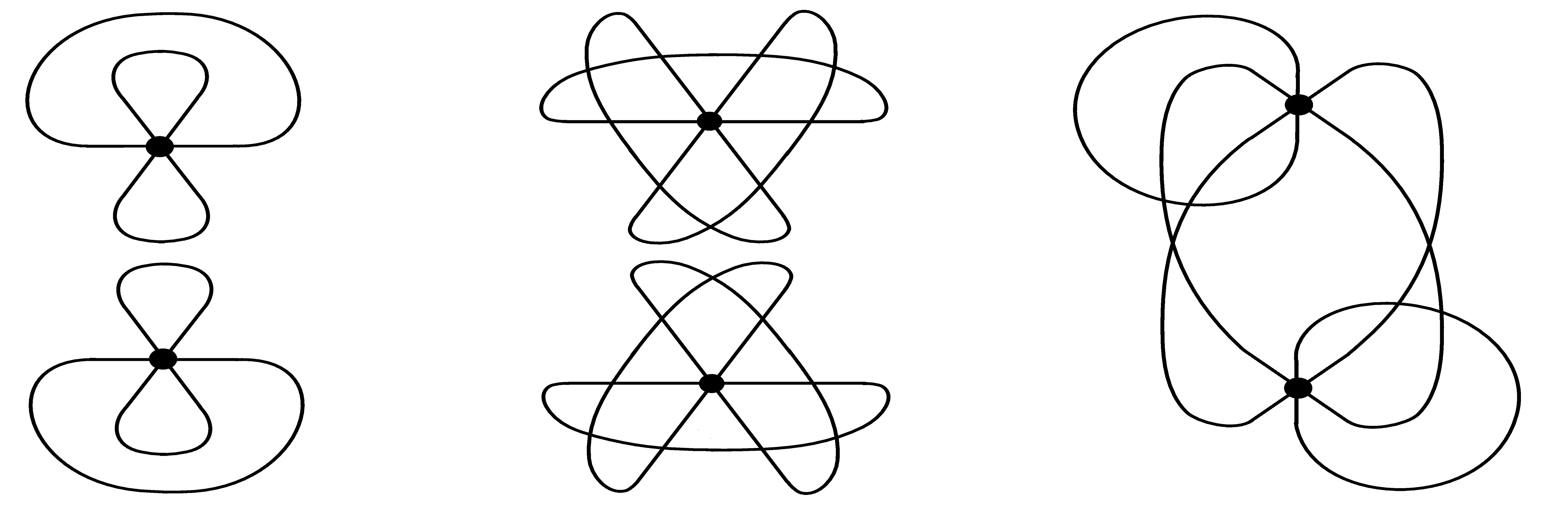}
\caption{From left to right: 4-cylinder disconnected separatrix diagram, 2-cylinder disconnected and 2-cylinder connected diagrams.}
\end{center}
        \label{Fig.Separatrix2}
	\end{figure}
More precisely,
\begin{itemize}
\item 4-cylinder disconnected diagram contributes $3E_2^2-\frac{1}{6}E_2+E_3-\frac{5}{6}E_4$,
\item 2-cylinder disconnected diagram contributes $\frac{1}{3}E_2-\frac{1}{2}E_3+\frac{1}{6}E_4$,
\item 2-cylinder connected diagram contributes $E_2-\frac{3}{2}E_3+\frac{1}{2}E_4$.
\end{itemize}
That gives in total $\sum_{n=1}^\infty N_n^\Real(2,2)q^n=3E_2^2+\frac{7}{6}E_2-E_3-\frac{1}{6}E_4$, where $E_k$ is the Eisenstein series defined above.\footnote{Note that disconnected separatrix diagrams can be associated with both connected and disconnected (2-component) real origami depending on how the cylinders are attached to them.} This calculation essentially coincides with that of \cite{DGM2021}, pp. 13--15, for the stratum $\mathcal{H}(2)$. However, such an
approach becomes overly complicated even for low-dimensional strata in the moduli spaces
of Abelian differentials.

	  \subsection{Quasimodularity vs. quantum modularity}
      As shown in \cite{EO2001}, the generating function for the numbers of complex origami
      in the stratum $\mathcal{H}(\mu_1-1,\ldots,\mu_k-1)$ of Abelian differentials with $k$ zeros of orders $\mu_1-1,\ldots,\mu_k-1$ is a quasimodular form of weight $2g$ (here $\mu=[\mu_1\ldots\mu_k]$
	is a partition of $2n$ encoding the ramification profile over $0\in T$). Although not 
    (quasi)modular, the Eisenstein series $E_3$ that appears in Theorem
    \ref{Th:Genfunction} and Subsection \ref{g3} belongs to the class of
    {\it quantum modular forms} introduced by D.~Zagier in \cite{Zagier2010}. 
    Explicit computations for these two strata (namely $\mathcal{H}(1,1)$ and $\mathcal{H}(2,2))$, as well as for some other low genus strata\footnote{To appear elsewhere.} make it natural to expect that
	the following is true:
	\begin{conjecture}\label{Conj:QMF}
		The generating function for the numbers of real origami in the stratum
		$\mathcal{H}(\mu_1-1,\mu_1-1,\ldots,\mu_k-1,\mu_k-1)$ of Abelian differentials with
        $\mu$ a partition of $n$ is a quantum modular form.
	\end{conjecture}

      	  We now dwell a little on Conjecture \ref{Conj:QMF} with a 
	  special case. In the following paragraphs we skip most of the details and refer the reader to the cited articles.
	  We call a ramified covering of the torus \emph{simple}
	  if the branch points are all simple. In particular,
	  by the Riemann-Hurwitz formula, a covering of the torus simply ramified over
        $2g-2$ points has genus $g$.
	  It was predicted by Dijkgraaf \cite{Dij}, and formally proved by Kaneko
	  and Zagier \cite{KZ}, that the generating function for the numbers of simply ramified genus $g$
	  coverings of the torus, denoted
	  here by $Z(q,x)$,  is a quasi-modular form of weight $6g-6$ for the modular group
	  $\text{SL}(2,\Integer)$. 
	  The starting point of this topic was the following important observation:
	  \begin{proposition}\label{Prop:Content}
	  	The number $T_{g,d}$ of (possibly disconnected) simply ramified coverings of the torus
	  	of genus $g$ and degree $d$ is given by
	  	\begin{equation}
	  		T_{g,d} = \sum_{\lambda\vdash d}\nu(\lambda)^{2g-2},
	  	\end{equation}
	  	where 
	  	 \begin{equation}\label{Eq:Content}
	  	\nu(\lambda) = \frac{1}{2}\sum_{i=1}^\infty(\lambda_i-i+1/2)^2-(-i+1/2)^2
	  	\end{equation}
	  \end{proposition}
	  
	  The quantity $\nu(\lambda)$ is important and appears
	  in different contexts; we will come back
	  to it later.
	  \iffalse has several incarnations depending on the context: 
	  the sum of \emph{content}
	  of the partition $\lambda$, the second \emph{shifted symmetric power sum} $\mathbf{p}_\lambda$,
	  or the central characters $\omega_1^\lambda$ of the symmetric
	  group, up to the coefficient $\frac{1}{2}$. We will come back to this later.
	  \fi
	  Using Proposition \ref{Prop:Content}, the generating function of possibly
	  disconnected ramified covering of the torus $Z^\circ(q,x)$ can be written as
	  $$Z^\circ(q,x)=\sum_{n=1}^\infty\sum_{\lambda\vdash n}e^{\nu(\lambda)x}q^{|\lambda|},$$
	  see \cite{Dij,KZ} for details.
	  
	  We now treat the real structure case. Denote by $\text{conj}:z\mapsto \bar{z}$ the complex conjugation on the torus $T$. We call a simply ramified covering $f:X\to T$
	  \emph{real} if the covering surface $X$ admits an orientation reversing
	  involution $\phi$ without fixed points such that 
	  $f\circ\phi(x)={\rm conj}\circ f(x)$ for any $x\in X$.
	  In particular if the covering surface has genus $g$
	  then the critical points come in $g-1$ pairs 
	  $(x_1,\phi(x_1),\dots,x_{g-1},\phi(x_{g-1}))$.
	  Denote by $T_{g,d}^\Real$ the number of simply ramified coverings of the torus
	  of degree $d$ and genus $g$, possibly disconnected. Notice that
	  the geometric definition of $T_{g,d}^\Real$ given here
	agrees (up to a simple renormalization) with the combinatorial Definition 3       of \cite{HM24} that follows from Proposition \ref{Prop:DefEqui}.
	  Now, using the combinatorial definition of $T_{g,d}^\Real$ of \cite{HM24}
	  and the results of \cite{Mat11} we get
	  \begin{proposition}\label{Prop:ContentReal}
	  	\begin{equation}
	  		T_{g,d}^\Real = \sum_{\lambda\vdash d}\nu_\Real(\lambda)^{g-1},
	  	\end{equation}
	  	where \begin{equation}\label{Eq:ContentReal}
	  	\nu_\Real(\lambda) = \sum_{i=1}^\infty(\lambda_i-i/2)^2-(-i/2)^2.
	  	\end{equation}
	  \end{proposition}
	  Let
	  \begin{equation}
	  	F_g(q) = \sum_{d=1}^\infty T^{\Real,\text{conn}}_{g,d}q^d
	  \end{equation}  
	 be the generating function for the numbers of genus $g$ real connected simply ramified coverings of the torus. Then we have  
	 \begin{theorem}
     For the generating function $F_2(q)$ holds the identity
	 	$$F_2(q)=\frac{1}{2}(E_3(q)-E_2(q)).$$ 
        In particular, $F_2(q)$ is a quantum modular form.
	 \end{theorem}
     \begin{proof}
     The proof is based on Lemma \ref{Lemma:MonoInvo} and is similar to the proof of Theorem \ref{Th:HsimisH} but for involutions without fixed points instead of
     general permutations. Alternatively, applying the propagator of \cite{HM24} to a loop with one vertex (which is
     the unique tropical graph of genus 2) we get the same result.
     \end{proof}
     \begin{remark}
     It is tempting to treat the equality 
     $$
    T^{\Real,\text{conn}}_{2,d}=N_d^\Real(1,1)
     $$
     as a tiny instance of the general ``mirror symmetry'' phenomenon.    
     \end{remark}
     
	Notice that the generating function $Z_\Real^\circ(q,x)$ of (possibly
	  disconnected) real ramified coverings of the torus  can be written as
$$
Z^\circ_\Real(q,x)=\sum_{n=1}^\infty\sum_{\lambda\vdash n}e^{\nu_\Real(\lambda)x}q^{|\lambda|}.
      $$
      
	 We now bring the quantities from Equation \ref{Eq:Content} and \ref{Eq:ContentReal}
	 into a broader context. We start with Equation \ref{Eq:Content}. 
     \iffalse
	 Consider in $Z\Complex[S_n]$ the basis of idempotent $F_\lambda$, and
	 by $C_2$ the class sum of transpositions in $Z\Complex[S_n]$, then we have
	 \begin{equation}
	 	C2.F_\lambda=\nu(\lambda)F_\lambda.
	 \end{equation}
	 \fi 
     The \emph{algebra of shifted symmetric polynomials} is defined as
     the projective limit 	 
$\Lambda^*=\lim\limits_{\longleftarrow}\Lambda^*(n)$ , where
	 $\Lambda^*(n)$ is the algebra of symmetric polynomials in $n$
	  variables $\lambda_1 -1 ,\dots, \lambda_n-n$.
	 The limit is taken with respect to the homomorphisms setting the last
	 variable equal to zero. In particular, $\nu(\lambda)$ belongs to $\Lambda^*$.
     
	 Introduce the family of shifted symmetric functions
	 \begin{equation}
	 \mathbf{P}_\ell(\lambda) = \sum_{i=1}^\infty(\lambda_i-i+1/2)^\ell-(-i+1/2)^\ell
	 \text{ and } \mathbf{P}_\lambda =\prod_{i}\mathbf{P}_{\lambda_i}\,,
	 \end{equation}
	 so that $\nu(\lambda)=\frac{1}{2}\mathbf{P}_2(\lambda)$.
	 Consider also the set of \emph{central characters} that are defined as
	 \begin{equation}\label{Eq:CentChar}
	 	f_\sigma(\lambda)=\frac{|C_\sigma|}{\chi^\lambda(1)}\chi^\lambda(\sigma),
	 \end{equation}
	 where $|C_\sigma|$ is the cardinality of the conjugacy class of $\sigma$,
	 $\chi^\lambda(1)$ is the dimension of the irreducible representation
	 indexed by $\lambda$, and $\chi^\lambda(\sigma)$ is the value of the character
     indexed by $\lambda$ applied to $\sigma$. We write $f_k(\lambda)=f_\sigma(\lambda)$ in case 
	 when $\sigma$ is a $k-$cycle. In particular, we have
	 $$f_2=\nu(\lambda),$$
	 see \cite{Mac95}, Chapter 1. 
     
	 We now summarize the main properties of shifted symmetric functions
in the following
	 \begin{theorem}[\cite{KO94}]\label{Th:KO}
	 	The algebra $\Lambda^*$ is freely generated by $\mathbf{P}_\ell$
	 	with $\ell\geq1$. The functions $f_\mu$ belong to $\Lambda^*$ and form a 
	    basis of $\Lambda^*$
	    as $\mu$ ranges over all partitions.
        % of $n$ ??
	 \end{theorem}
	 
	 We now recall a remarkable theorem of Bloch and Okounkov.
	 Using Theorem \ref{Th:KO} one provides the algebra $\Lambda^*$
	 with the weight grading by assigning to $ \mathbf{P}_\ell$
	 the weight $k=\ell +1$. Let $f: \text{Part}(n)\to\mathbb{Q}$
	 be an arbitrary function on the set of all partitions.
     % of $n$ ??
     Its $q$-bracket is defined as the formal power series 
	 $$\langle f \rangle_q =\frac{\sum_\lambda f(\lambda)q^{|\lambda|}}{\sum_\lambda q^{|\lambda|}}
	 \in \mathbb{Q}[[q]]\;. $$
	 
	 \begin{theorem}[\cite{BO00}]\label{Th:OkBlo}
	 	If $f$ is a shifted symmetric function of weight $k$, then
	 	$\langle f \rangle_q$ is a quasi-modular form of weight $k$. 
	 \end{theorem}
	 
	 Similarly to the algebra of shifted symmetric polynomials,
	 we can define the algebra of \emph{$\alpha$-shifted symmetric polynomials}
	 $\Lambda^*_\alpha$.  The algebra of $\alpha$-shifted symmetric polynomials is defined as
	 $\Lambda^*_\alpha=\lim\limits_{\longleftarrow}\Lambda^*_\alpha(n),$ where
	 $\Lambda^*_\alpha(n)$ is the algebra of symmetric polynomials in the $n$
	 variables $\lambda_1 -\frac{1}{\alpha} ,\dots, \lambda_n-\frac{n}{\alpha}$.
	 The projective limit is taken with respect to the homomorphisms setting the last
	 variable equal to zero. In particular $\nu_\Real(\lambda)$ belongs to $\Lambda^*_2$.
	 To continue the parallel with the classical theory,
	 we can introduce the family of 2-shifted symmetric functions
	 \begin{equation}
	 	\mathbf{R}_\ell(\lambda) = \sum_{i=1}^\infty(\lambda_i-i/2)^\ell-(-i/2)^\ell
	 	\text{ and } \mathbf{R}_\lambda =\prod_{i}\mathbf{R}_{\lambda_i}
	 \end{equation}
	 so that $\nu_\Real(\lambda)=\mathbf{R}_2(\lambda)$.
	 
	 We also introduce the \emph{central zonal function} that are defined as
	 \begin{equation}\label{Eq:CentChar}
	 	f^\Real_\sigma(\lambda)={|K_\sigma|}\omega^\lambda(\sigma).
	 \end{equation}
	 Here $|K_\sigma|$ is the cardinality of the double coset class of $\sigma$
	 in $H_n\backslash S_{2n}/H_n$, $H_n$ is the hyperoctahedral
	 group, and $\omega^\lambda(\sigma)$ is the zonal function indexed
	 by $\lambda$ applied to $\sigma$ (see the next section for details).
	 We write $f_k^\Real(\lambda)$ in the case 
	 where $\sigma$ is a $k-$cycle. In particular we have 
	 $$f^\Real_2=\nu_\Real(\lambda),$$
	 see \cite{Mac95} Chapter VII.
     
	 Similarly to Theorem \ref{Th:KO}, we have for $2$-shifted symmetric functions (in fact, for $\alpha$-shifted symmetric functions)
	 \begin{theorem}[\cite{Las08}]\label{Th:Lass}
	 	The algebra $\Lambda^*_2$ is freely generated by the $\mathbf{R}_\ell$
	 	with $\ell\geq1$. The functions $f^\Real_\mu$ belong to $\Lambda^*_2$ and form a basis of $\Lambda^*_2$
	 	as $\mu$ ranges over all partitions.
	 \end{theorem}

     It is also worth mentioning that the quantities $\nu(\lambda)$
     and $\nu_\Real(\lambda)$ are parts of a family of eigenvalues
     of the Beltrami-Laplace operator, see \cite{Mac95} chapter
     VI.
     
	 Given the striking resemblance between the worlds of
     complex and real ramified coverings of the torus, together with the fact 
	 that $F_2(q)$ is quantum modular\footnote{Numerical evidence suggests that $F_3(q)$ is also
     quantum modular.}, it is natural to expect that the following analog  
     of Theorem \ref{Th:OkBlo} holds:
     \begin{conjecture}
	 	If $f$ is a 2-shifted symmetric function of weight $k$, then
	 	$\langle f \rangle_q$ is a quantum modular form of weight $k$. 	
     \end{conjecture}
	 This conjecture currently has a status of work in progress.
	  
	\subsection{Involutions without fixed point and the Gelfand pair $(S_{2n},H_n)$}
	Denote by $\mathcal{O}_\Real^{\circ}(\lambda)$ 
	the number of real origami in the sense of (slightly modified) Definition \ref{Def:OrigamiCombi},
	where
	\begin{itemize}
		\item  $\tau:=(1\,\bar 1)\dots (n\,\bar n)$ is \emph{fixed},
		\item transitivity condition \emph{not required},
		\item $\lambda=[\lambda_1\ldots\lambda_s]$ is a partition of $n$ such that the commutator
        $[h,v]$ has cycle structure $[\lambda_1^2\ldots\lambda_s^2]$, see Remark \ref{Rem:Profil}.
	\end{itemize}

    Our goal for the rest of this section is to prove the identity
	
	\begin{equation}
		\sum_{\rho\vdash n} Z_\rho = \frac{1}{2^nn!}\sum_{\lambda\vdash n}
		\mathcal{O}_\Real^{\circ}(\lambda) p_\lambda\,,
	\end{equation}
	where $Z_\rho$ are zonal polynomials and $p_\lambda$ are (products of) power sums.
	
        We start by recalling the necessary results from \cite{BF24}.
	Let $S_{2n}$ be the symmetric group on $2n$ elements $\{1,\bar1,\dots,n,\bar n\}$
	with ordering $1<\bar1<\dots<n<\bar n$, and $\tau$ is fixed as above:
	\begin{equation}\label{Eq:tau}
		\tau = (1,\bar1)\dots(n,\bar n).
	\end{equation}
	
	\begin{definition}
		The hyperoctahedral group $H_n$ is the centralizer of $\tau$ in the 
		symmetric group $S_{2n}$.
	\end{definition}
    In other words, $H_n$ consists of the elements $\sigma\in S_{2n}$ satisfying
		$\tau\sigma\tau\sigma^{-1}=id$. The order of $H_n$ is $2^nn!$.
	
	\begin{example}
		Some elements of $S_6$ belonging to $H_3$:
		\begin{equation}
			(1,2,3)(\bar 1,\bar 2,\bar 3);\ (1,2,3,\bar 1,\bar 2,\bar 3);\ 
			(1,2)(\bar 1,\bar 2)(3)(\bar 3); (1,3,2)(\bar 1, \bar 3, \bar 2),(3\bar 3)\dots
		\end{equation}
	\end{example}
	
	Consider now the subset $C^{\sim}(\tau) = \{\sigma \in S_{2n} \mid \tau \sigma =
	\sigma^{-1} \tau\}$ of $S_{2n}$.
	
	\begin{lemma}[see \cite{BF24} for details]
		Let $\sigma = c_1 \dots c_m \in C^{\sim}(\tau)$, where $c_1, \ldots, c_m$ are
independent cycles. Then for every $i$
		\begin{itemize}
			\item either there exists $j \ne i$ such that $c_i = (u_1 \ldots
			u_k)$ and $c_j = (\tau u_k \ldots \tau u_1)$,
                \item or $c_i$ has even length $2k$ and looks like $c_i = (u_1 \ldots
			u_k\, \tau u_k \ldots \tau u_1)$; 
		\end{itemize}
        here $\{u_1,\ldots,u_{2n}\}$ is a set of elements permuted by $S_{2n}$.
	\end{lemma}
	
	In the first case we say that the cycles $c_i$ and $c_j$ are
	$\tau$-symmetric, and in the second case the cycle $c_i$ is
	$\tau$-self-symmetric.
	
    \begin{definition}
    \label{It:AlphaCycl} Denote by $B_n^\sim$ the set of permutations
    $\sigma \in C^{\sim}(\tau)$ such that their cycle decomposition
    contains no $\tau$-self-symmetric cycles.
    \end{definition}

The properties of $B_n^\sim$ that we need are summarized in the following two propositions.
	
\begin{proposition}\label{Prop:PropertiesBn} 
    Any element $\sigma\in B_n^\sim$ satisfies $\sigma\tau\sigma\tau =id$.
    In particular, both $\sigma\tau$ and $\tau\sigma$ are involutions \emph{without fixed points}. Moreover, each of the following two sets is in bijection with $B_n^\sim$: 
    \begin{enumerate}
%\item the set of $\tau$-symmetric permutations (i.e. such that any element has the form $c_1\dots c_j\tau c_j\dots c_1\tau$, where for each $c_\ell = (u_1 \dots u_k)$, $u_i\neq \tau u_j$ for all $i,j$;
    \item\label{It:Coset} 
    the set $S_{2n}/H_n$ of right cosets of $H_n$ in $S_{2n}$;
    \item\label{It:Invol} the set of fixed-point-free involutions $\iota \in S_{2n}$.
    \end{enumerate}
    In particular, all these sets have cardinality $(2n-1)!!$.
	\end{proposition}
   
	\begin{example}
		Here are a few permutations belonging to $B_3^\sim$:
		\begin{equation}\label{Ex:Bntild}
			(1,2,3)(\bar 3,\bar 2,\bar 1);\ 
			(1,2)(\bar 1,\bar 2)(3)(\bar 3); (1,3,2)(\bar 2, \bar 3, \bar 1),\dots
		\end{equation}
		However, the following permutations \emph{do not} belong to $B_3^\sim$:
		\begin{equation}
			(1,2,\bar 2)(\bar 2, 2,\bar 1);\ 
			(1,\bar 1)(2,\bar 2)(3)(\bar 3); (1,\bar 3,3)(\bar 3,  3, \bar 1),\dots
		\end{equation}
	\end{example}

Fix a partition $\lambda\vdash n,\;\lambda=[\lambda_1 \ldots \lambda_s]$, and denote by $B_\lambda^{\sim} \subset B_n^{\sim}$
	the set of permutations whose decomposition into independent cycles
	consists of $s$ pairs of $\tau$-symmetric cycles of lengths
	$\lambda_1 \dots, \lambda_s$. Apparently, $B_n^{\sim} =
	\sqcup_{\lambda\vdash n} B_\lambda^{\sim}$.
	
	\begin{proposition}\label{Prop:ConjClass}
		$B_\lambda^{\sim}$ is an $H_n$-conjugacy class in $S_{2n}$.
	\end{proposition}
	
	\begin{example}
		
		In Example (\ref{Ex:Bntild}), we see that the first and last
		permutations belong to $B_{[3]}^\sim$, while the second one
		belongs to $B_{[2,1]}^\sim$.
		Furthermore, to illustrate Proposition \ref{Prop:PropertiesBn}, we notice that
		$\sigma=(1,2,3)(\bar 3,\bar 2,\bar 1)\in B_{[3]}$,
		and $\sigma\tau=(1\bar2)(2\bar3)(3\bar1)$ is an involution
		without fixed points.
	\end{example}
	
	Our next step will be to express the commutator
	\begin{equation}
		[v,h]= vhv^{-1}h^{-1}
	\end{equation}
	where $(v,h)\in B_n^\sim\times H_n$, in terms of certain 
	connection coefficients in the group algebra $\mathbb{C}[S_{2n}]$. This makes sense since 
        we conjugate an element of $B_n^\sim$ with an element of $H_n$, see Proposition \ref{Prop:ConjClass}). Thus, if $v\in B_\lambda^\sim$,
	then $hv^{-1}h^{-1}$ is also in $ B_\lambda^\sim$.
	
	Rewrite the commutator as follows:
        \begin{equation}\label{Eq:Commutator}
		[v,h]= (v\tau)(\tau hv^{-1}h^{-1}).
	\end{equation}
	
	We see that because $v$ and $hv^{-1}h^{-1}$ belong to $B_n^\sim$,
	then $(v\tau)(\tau hv^{-1}h^{-1})$ is just the product of two 
	involutions without fixed points by Proposition \ref{Prop:PropertiesBn}. 
	Furthermore, we can rewrite 
	\begin{equation}
		\tau hv^{-1}h^{-1} = h\tau v^{-1}h^{-1} = h v\tau h^{-1}\,.
	\end{equation}
	
        As we will see soon, the involutions without fixed points are closely related
	to the Hecke algebra of the Gelfand pair $(S_{2n},H_n)$. 
	Following \cite{GJ96} (see also \cite{Mac95}, Chapter VII, for details),
	   consider the set $\mathcal{P}_n=\{\iota_1,\dots,\iota_{(2n-1)!!}\}$ 
	of perfect matchings  on the set $\{1,\bar1\dots,n,\bar n\}$. 
	Denote by $G(\iota_i,\iota_j)$ the graph with vertices $\{1,\bar1,\dots n,\bar n\}$,
	whose edges are formed by $\iota_i$ and $\iota_j$.
	Each of the $s\geq 1$ connected components of $G(\iota_i,\iota_j)$ contains an even number of         edges $2\lambda_k,\; k=1,\ldots, s$, arranged in weakly decreasing order; see Fig.                   \ref{Fig.Perf_Match}.

\begin{figure}[h]
		\includegraphics[scale=.6]{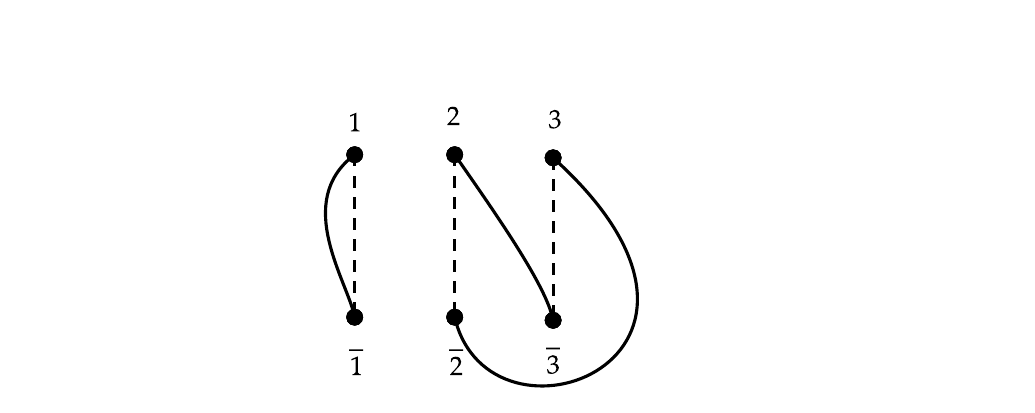}  
		\caption{Two perfect matchings:
			$\iota_1=(1,\bar1)(2,\bar2)(3,\bar3)$ (dashed lines) and
			$\iota_2=(1,\bar1)(2,\bar3)(3,\bar2)$ (solid lines). Here
			$\Lambda(\iota_1,\iota_2)=[2,1]$. As permutations,
			$\iota_1\iota_2 = (1)(\bar{1})(2,3)(\bar{2},\bar{3})$, 
			with cycle decomposition $[2^2,1^2]$.}\label{Fig.Perf_Match}
	\end{figure}
	
	Denote by $\Lambda(\iota_i,\iota_j)=[\lambda_1\ldots\lambda_s]$ the
        corresponding partition of $n$.

        Equivalently, a perfect matching $\iota_k$ can be thought of as an involution in $S_{2n}$ without fixed points that we will denote (slightly abusing notation) by the same symbol. In this case, the product $\iota_i\iota_j$ has a cycle decomposition of the form $[\lambda_1^2\ldots\lambda_s^2]$ with the same entries $\lambda_k$ as above.
        We will freely use both of these approaches.
    
	The commutator in (\ref{Eq:Commutator}) can be decomposed as follows.
	Fix perfect matchings $v\tau$ and 
	$\iota$ such that
	$\Lambda(v\tau,\iota)=\mu$,
	with $\mu\vdash n$, and fix an element $h\in H_n$.
	Introduce the numbers  
	\begin{gather}
		\kappa^\mu_{\mu,\lambda} = 2^n n!\,\#\,\{hv\tau h^{-1} \in \mathcal{P}_n\ | \ 
		\Lambda(hv\tau h^{-1},\iota)=\mu\ \text{and}\ 
		\Lambda(v\tau,hv\tau h^{-1})=\lambda\},
	\end{gather}
	where $\lambda$ and $\nu$ are partitions of $n$, and the factor $2^nn!$ is the order of $H_n$.
    In particular,
	\begin{equation}\label{Eq:HurwandCoeffk}
		\mathcal{O}_\Real^{\circ}(\lambda) = \sum_\mu \kappa^\mu_{\mu,\lambda}\,,
	\end{equation}
	where $\mu$ runs over all partitions of $n$.
    
	The numbers $\kappa^\mu_{\mu,\lambda}$ admit an explicit expression
	in terms of connection coefficient in a double coset algebra that we describe now.
	Recall that $H_{n}$ is a subgroup of $S_{2n}$ of order $|H_n|=2^nn!$.
	We say that a permutation $\sigma\in S_{2n}$ belongs to 
	the double coset $\mathcal K_\lambda\in H_n\backslash S_{2n}/H_n$ if $\sigma\tau\sigma^{-1}\tau$      has cycle structure
	$\lambda=[\lambda_1\lambda_1\ldots\lambda_s\lambda_s]$.
	The double cosets $\mathcal K_\lambda$ are naturally indexed by partitions $\lambda\vdash n$.
	\begin{remark}\label{Rem:Size}
		It is important to recall that we defined $B_\lambda^\sim$ as the set of permutations 
		$\sigma \in B_n^\sim$ having cycle decomposition $\lambda$. Since
		$\mathcal K_\lambda$ consists of \emph{all} permutations $\sigma\in S_{2n}$ such that 
            $\sigma\tau\sigma^{-1}\tau$ has cycle  
		decomposition $\lambda$, this implies $|\mathcal{K}_\lambda|=|B^\sim_\lambda|\cdot |H_n|$.
	\end{remark}
	
	Define $K_\lambda \in \Complex[S_{2n}]$ as the formal sum of
	all elements of $\mathcal{K}_\lambda$. Then $K_\lambda$
	form a basis of a commutative subalgebra of $\Complex[S_{2n}]$,
	known as the double coset algebra, that identifies with the
	Hecke algebra of the Gelfand pair $(S_{2n},H_n)$.
	
	Denote by $b^\mu_{\mu,\lambda}=[K_\mu]K_\mu K_\lambda$ the coefficient at $K_\mu$
	in the linear expansion of $K_\lambda K_\mu$ 
	with respect to the basis $K_\lambda$.
	
	The reason for introducing this machinery is due to the
	following lemma:	
	
	\begin{lemma}[\cite{HSS92}, Lemma 3.2] \label{Lem:HeckeAndMatching}
        We have
		\begin{equation}
			b^\mu_{\mu,\lambda} = \kappa^\mu_{\mu,\lambda}\,.
		\end{equation}	
	\end{lemma}
	
	Recall a particular case of a ``real'' analogue of the Frobenius formula.
	
	\begin{lemma}[\cite{HSS92}, Lemma 3.3]\label{Lem:RealFrob}
		\begin{equation}
			b^\mu_{\mu,\lambda} = \frac{|K_\mu||K_\lambda|}{(2n)!}\sum_{\rho\vdash n}
			\chi^{2\rho}(1)\omega^\rho(\lambda)\omega^\rho(\mu)\omega^\rho(\mu)\,,
		\end{equation}
		where $2\rho=[2\rho_1\ldots2\rho_m]$ is a partition of $2n$ and $\chi^{2\rho}(1)$ 
		is the dimension of the irreducible
		representation $\chi^{2\rho}$ of $S_{2n}$. In addition, 
		$\omega^\rho(\mu):=\omega^\rho(x)$ is the value of 
		the zonal spherical function $\omega^\rho$ on any $x$ in $\mathcal K_\mu$
		with $\omega^\rho(x) = \frac{1}{2^nn!}\sum_{h\in H_n}\chi^{2\rho}(xh)$.
	\end{lemma}
	
	We now properly define the zonal polynomials
	
	\begin{definition}\label{Def:ZonPoly}
		For $\lambda=[1^{m_1}2^{m_2}\ldots]$ a partition of $n$ put
		$z_\lambda := \prod_i i^{m_i}m_i!$.
		The zonal polynomials $Z_\rho$ are defined by the formula
		\begin{equation}
			Z_\rho= |H_n|\sum_{\lambda \vdash n}
			\frac{1}{z_{2\lambda}}\omega^\rho(\lambda)p_\lambda\,,
		\end{equation}
        where $p_\lambda=\prod_i p_i^{m_i}$ and $p_i$ is the $i^{\text th}$ power sum.
	\end{definition}
	
	We now have all the ingredients to prove the following
	
	\begin{theorem}\label{Prop:SumZonal}
		\begin{equation}\label{Eq:ZonalHurw}
			\sum_{\rho\vdash n} Z_\rho = \frac{1}{2^nn!}\sum_{\lambda\vdash n}
			\mathcal{O}_\Real^{\circ}(\lambda) p_\lambda.
		\end{equation}
	\end{theorem}
	
	\begin{proof}
		Expand the right-hand side of \eqref{Eq:ZonalHurw}:
		\begin{equation}
			\begin{split}
				&\frac{1}{2^nn!}\sum_{\lambda\vdash n}
				\mathcal{O}_\Real^{\circ}(\lambda) p_\lambda\\
				&=\frac{1}{2^nn!}\sum_{\lambda\vdash n}\sum_{\mu\vdash n} \kappa^\mu_{\mu,\lambda}
				p_\lambda\qquad \text{(by Eq. \eqref{Eq:HurwandCoeffk})}\\
				&=\frac{1}{2^nn!}\sum_{\lambda\vdash n}\sum_{\mu\vdash n}\frac{|K_\mu||K_\lambda|}{(2n)!}\sum_{\rho\vdash n}
				\chi^{2\rho}(1)\omega^\rho(\lambda)\omega^\rho(\mu)\omega^\rho(\mu)p_\lambda\qquad \text{(by Lemma \ref{Lem:RealFrob})}\label{zonal}
			\end{split}
		\end{equation}
		
		Recall that the  Hecke algebra of $(S_{2n},H_n)$ has the scalar product
		\begin{equation}
			\langle f,g\rangle=\sum_{\sigma\in S_{2n}}f(\sigma)g(\sigma)
		\end{equation}
		and for zonal spherical functions we have		
		\begin{equation}
			\langle \omega^\lambda,\omega^\rho \rangle = 
			\delta_{\lambda,\rho}\frac{(2n)!}{\chi^{2\lambda}(1)}.
		\end{equation}
		We continue to get
		\begin{equation}
			\begin{split}
				&=\frac{1}{2^nn!}\sum_{\lambda\vdash n}\sum_{\mu\vdash n}\frac{|K_\mu||K_\lambda|}{(2n)!}\sum_{\rho\vdash n}
				\chi^{2\rho}(1)\omega^\rho(\lambda)\omega^\rho(\mu)\omega^\rho(\mu)p_\lambda\\
				&=\frac{1}{2^nn!}\sum_{\lambda\vdash n}\frac{|K_\lambda|}{(2n)!}\sum_{\rho\vdash n}
				\chi^{2\rho}(1)\omega^\rho(\lambda)
				\underbrace{\sum_{\mu\vdash n}|K_\mu|\omega^\rho(\mu)\omega^\rho(\mu)}_{
					=\langle \omega^\rho,\,\omega^\rho \rangle}p_\lambda\\
				&=\frac{1}{2^nn!}\sum_{\rho\vdash n}\sum_{\lambda\vdash n}
				|K_\lambda|\omega^\rho(\lambda)p_\lambda\,.
			\end{split}
		\end{equation}
		%For $\mu$  a partition $(1^{m_1},2^{m_2},\dots)$ denote $z_\mu := \prod i^{m_i}m_i!$.
		%Furthermore the size of the double coset $K_\lambda$ is given by 
            From Definition \ref{Def:ZonPoly} of zonal polynomials with the help of the formula
            \begin{equation}
			|K_\lambda|=\frac{|H_n|^2}{z_{2\lambda}}
		\end{equation}
		(see \cite{Mac95}, Chapter VII, Section 2)
		we eventually get		
		\begin{equation}
			\begin{split}
				&=\frac{1}{2^nn!}\sum_{\rho\vdash n}\sum_{\lambda\vdash n}
				|K_\lambda|\omega^\rho(\lambda)p_\lambda\\
				&=\frac{1}{|H_n|}\sum_{\rho\vdash n}\sum_{\lambda\vdash n}
				\frac{|H_n|^2}{z_{2\lambda}}\omega^\rho(\lambda)p_\lambda\\
				&=\sum_{\rho\vdash n} Z_\rho
			\end{split}
		\end{equation}
		completing the proof.
	\end{proof}
	
	To finish this section consider the polynomials 
	\begin{equation}
	   P_n^\circ(\mathbf{p}):=\frac{1}{2^nn!}\sum_{\lambda\vdash n}
	   \mathcal{O}_\Real^{\circ}(\lambda) p_\lambda \,,
	\end{equation}
        where $\mathbf{p}=(p_1,p_2,\ldots)$, and arrange them into the generating function
	\begin{equation}\label{Eq:GenFunctR}
		\mathcal{P}^\circ(\mathbf{p},t) = \sum_nP_n^\circ(\mathbf{p})\, t^n.
	\end{equation}
	This generating function enumerates \emph{all} real origami that are not necessarily
        connected. To get the generating function counting only \emph{connected} real origami 
        (as claimed in Definition \ref{Def:OrigamiGeo}) it is sufficient to take the logarithm:	
	\begin{equation}\label{Eq:GenFunctConnR}
		\mathcal{P}(\mathbf{p},t):=\log\mathcal{P}^\circ(\mathbf{p},t)\,.
	\end{equation}

    Denote by $P_n=P_n(\mathbf{p})$ the coefficient of $\mathcal{P}(\mathbf{p},t)$ at $t^n$.
	Then the coefficient of $P_n$ by
	  the monomial $p_\lambda,\;\lambda\vdash n$, gives the number
	of \emph{connected} real origami of degree $2n$ with ramification profile 
	$\lambda = [\lambda_1\lambda_1\lambda_2\lambda_2\ldots]$ (counted with weights reciprocal to
       the orders of their automorphism groups).
    
    An advantage of Theorem \ref{Prop:SumZonal} is that it provides
	a reasonably fast algorithm for counting origami. 
	Appendix A contains the list of polynomials $P_n$ up to $n=13$ (i.e. up to degree $26$).
	
	\subsection{Yet another ``real''  structure}
	
	In this subsection we introduce an alternative to real origami that is associated with the mirror symmetric map
$\text{symm}:T\to T,\;(x,y)\mapsto (y,x).$
	
\begin{definition}[geometric]\label{Def:MobGeo}
		An origami $f:X\to T$ is called \emph{mirror symmetric}, or simply \emph{mirror}, if the following conditions
		are satisfied:
		\begin{itemize}
			\item $f:X\to T$ is a ramified cover of $T$ branching over $0\in T$; 
			\item there exist an anti-holomorphic fixed point free involution $\phi:X\to X$ 
			such that $f\circ\phi={\rm symm}\circ f$ and the factor space $X/\text{symm}$ is connected;
			\item the preimage of the parallel $\alpha=\{y=1/2\}\subset T$ is exchanged by $\phi$ with the preimage of the meridian $\beta=\{x=1/2\}\subset T.$
		\end{itemize}
	\end{definition}
	
	\begin{definition}[combinatorial]\label{Def:MobCombi}
		A pair of permutations $h,v\in S_{2n}$ defines a mirror origami
		if the following conditions are satisfied
		\begin{itemize}
                \item there exist an involution without fixed points $\tau\in S_{2n}$ such that 
			$h=\tau v \tau.$
			\item the group generated by $h,v$ and $\tau$ acts transitively on the set
			$\{1,\bar{1},\ldots,n,\bar{n}\}$; 
			
		\end{itemize}
	\end{definition}
	
	Following the same lines as in the proof of Proposition \ref{Prop:DefEqui}, 
	 we obtain
	\begin{proposition}
		Definitions \ref{Def:MobGeo} and \ref{Def:MobCombi} are equivalent.
	\end{proposition}
	
		Denote by $\mathcal{O}_{\text{mirr}}^{\circ}(\lambda)$ 
	the number of mirror origami in the sense of (slightly modified) Definition \ref{Def:MobCombi},
	where
	\begin{itemize}
		\item  $\tau:=(1\,\bar 1)\dots (n\,\bar n)$ is \emph{fixed},
		\item transitivity condition \emph{not required},
		\item $\lambda=[\lambda_1\ldots\lambda_s]$ is a partition of $n$ such that the commutator
		$[h,v]$ has cycle structure $[\lambda_1^2\ldots\lambda_s^2]$, see Remark \ref{Rem:Profil}.
	\end{itemize}
	
	We now introduce several technicalities in order to
	establish a bijection between the sets $\mathcal{O}_{\text{mirr}}^{\circ}(\lambda)$ 
	and $\mathcal{O}_{\Real}^{\circ}(\lambda)$.
	We start with a reformulation of Proposition 3 from \cite{ZJ09} (cf. also Lemma \ref{Lem:Monot}).
	
	\begin{lemma}\label{Lemma:MonoInvo}
		If $\iota$ is an involution without fixed points such
		that $\iota\tau$ has cycle decomposition $[\lambda_1^2,\dots,\lambda^2_s]$,
		then there exists a \emph{unique} sequence of $n-s$ transpositions
		$\sigma_1,\dots,\sigma_{n-s}$ such that
		\begin{itemize}
			\item for all $i\in \{1,\dots,n-s\}$, $\sigma_i$ is either of
			the form $(a_i,b_i)$ or $(\bar a_i,b_i)$ with $b_i\in\{1,\dots,n\}$ and
			$a_i<b_i$ (\text{remark that} $\bar a_i<b_i$),
			\item $b_i<b_{i+1}$ for $i=1,\dots,n-s$, 
			\item $\iota = \sigma_1\dots\sigma_{n-s}\tau\sigma_{n-s}\dots\sigma_{1}$.
		\end{itemize}
	\end{lemma}
	
	Notice that if $\zeta\in H_n$ then 
	$$\sigma_1\dots\sigma_{n-s}\zeta\tau \zeta^{-1}\sigma_{n-s}\dots\sigma_{1}=
	\sigma_1\dots\sigma_{n-s}\tau \sigma_{n-s}\dots\sigma_{1}$$ since
	$\zeta$ and $\tau$ commute by definition. 
	
	By Proposition \ref{Prop:PropertiesBn} the product $\sigma_1\dots\sigma_{n-s}$ is a representative
	of the right coset $S_{2n}/H_n$. Furthermore, the products
	$\sigma_1\dots\sigma_{n-s}\zeta$ for $\zeta\in H_n$ and
	$\zeta\neq id$ are the remaining elements in this coset.
	
	Similar to Lemma \ref{Lemma:MonoInvo}, any permutation $\alpha\in S_{2n}$ can be \emph{uniquely} 
	decomposed into a product of transpositions with an element $\zeta\in H_n$ so
	that $\alpha=\sigma_1\dots\sigma_{n-s}\zeta$. We put $\phi_v(\alpha)=\sigma_1\dots\sigma_{n-s}$
	and $\phi_h(\alpha)=\zeta$.
	
	Recall that for each element $\sigma\in B_n^\sim$,
	we can associate a unique involution without fixed points $\iota=\sigma\tau$.
	
	To fix the notation, let $h,v$ be such that $[v,h]\in\mathcal{O}_{\Real}^{\circ}(\lambda)$,
	and let $\xi,\eta$ be such that $[\eta,\xi] \in \mathcal{O}_{\text{mirr}}^{\circ}(\lambda)$.
Remark that
$$[\eta,\xi]=\eta\xi\eta^{-1}\xi^{-1}=\eta\tau \eta\tau \eta^{-1}\tau \eta^{-1}\tau
	=(\eta\tau \eta)\tau (\eta^{-1}\tau \eta^{-1})\tau,$$
	once again giving a commutator of two involutions without fixed points.
% In particular everything depends on $\eta$.
	Furthermore, Equation \ref{Eq:Commutator} and the discussion
	after it imply $[v,h]=[v\tau,h]$.

	Define the map $\Xi$ by the formula
	\begin{align}\label{Eq:BijMobReal}
			\Xi:S_{2n}\times S_{2n}&\rightarrow B_n^\sim\times H_n,\\
			(\eta,\xi)&\mapsto (\phi_v(\alpha)\tau(\phi_v(\alpha))^{-1}\tau,\phi_h(\alpha)).
	\end{align}
	The above results yield	
	\begin{proposition}
		The map $\Xi$ is a bijection between $\mathcal{O}_{\text{mirr}}^{\circ}(\lambda)$ and
		 $\mathcal{O}_{\Real}^{\circ}(\lambda)$.
	\end{proposition}
	\begin{proof}
		The only point left to verify is the existence of the inverse map to $\Xi$.
		Indeed, to a pair $(v,h)\in B_n^\sim\times H_n$ we associate
		an involution $v\tau$ without fixed points, cf. Equation \ref{Eq:Commutator}. Then use the decomposition
		of Lemma \ref{Lemma:MonoInvo}, and map $(v,h)$ to 
		$(\sigma_1\dots\sigma_{k}h,\tau \sigma_1\dots\sigma_{k}h\tau)$.
	\end{proof}
	
	\begin{remark}
    Here we considered two kinds of real origami
    structures associated with two anti-holomorphic involutions on the torus, whose factors were a cylinder and a M\"obius band. The third anti-holomorphic involution is fixed point free and factorizes the torus to the Klein bottle.
    However, it has no independent interest since this case reduces to the case of complex origami enumeration.
	\iffalse	We have enumerate origami with two possible different real structures,
		one real structure with two fixed locus on the torus, the real origami, and
		one real structure with one fixed locus on the torus, Möbius origami.
		A natural question would be how to enumerate origami with real structure having
		no fixed locus. In turns out that this question is of limited interest,
		as it is equivalent to enumerate covering of the Klein bottles. In turns it is
		equivalent to enumerate products of permutations of the form $vhvh^{-1}$ 
		which are in bijection with $vhv^{-1}h^{-1}$, classical origami.
	\fi
    \end{remark}

    \begin{remark}
    Although real and mirror origami are in one-to-one correspondence, their vertical separatrix diagrams are different.
    For example, every real origami in the stratum $\mathcal{H}(1,1)$ has vertical separatrix diagram of type IIa, see Figure \ref{Fig.Separatrix1} and the discussion thereafter.
\begin{table}[h]\label{mirr}\caption{Number of mirror origami in the stratum $\mathcal{H}(1,1)$ for each degree and vertical separatrix diagram.} 	\begin{tabular}{|l|r|r|r|r|r|}
				\hline
type$\backslash$ degree &4&6&8&10&12\\ \hline
1-cyl type I&0&1/2&3/2&3&7/2\\\hline
2-cyl type IIa&0&1/2&1&3&7/2\\\hline
2-cyl type IIb&1&2&4&4&11\\\hline
3-cyl type III&0&0&1/2&0&1\\\hline
Total&1&3&7&10&19\\\hline
                \end{tabular}
                \end{table}
    At the same time, the vertical separatrix diagram of a mirror origami in $\mathcal{H}(1,1)$ can be any out of 4  admissible ones. Their distribution for the small values of the degree is displayed in Table \ref{mirr}.
    \end{remark}
       
       	\section{Complex origami and quasimodularity}
	In this section we provide a relatively fast algorithm for enumeration of complex origami in the spirit
	of Theorem \ref{Prop:SumZonal} \footnote{What
		we call complex origami are the ordinary ones, but we emphasize their
		difference from the ``real'' origami introduced earlier in the article.}.
	We then express complex origami in terms of strictly monotone
	double Hurwitz numbers, and recover the quasimodularity of their generating function.
	
	\subsection {Complex origami}\label{Section:ComplexOrigami}
	
	\begin{definition}\label{Def:ComplexOrigami}
		An origami $\mathcal{O}_{\lambda,n}$ is a (possibly disconnected) finite cover 
		$f: X\to\ \Complex/(\Integer\oplus i\Integer)$ of degree $n$,
		branched only over the origin $0$ with ramification profile $\lambda,\;\lambda\vdash n$.
        
		Equivalently, an origami $\mathcal{O}_{\lambda,n}$ is a pair of 
		permutations $(\sigma,\rho)$ in $S_n$ such that
		the commutator $[\sigma,\rho ]= \sigma\rho\sigma^{-1}\rho^{-1}$ belongs
		to the conjugacy class $C_\lambda$ in $S_n$;
		see \cite{Ker09} for details.
	\end{definition}
	
	We are interested in the enumeration of isomorphism classes of origami, and
	to this purpose define the following numbers:
	
	\begin{gather}
		h_{\lambda,n}^\circ= \frac{1}{n!}\#\{\mathcal{O}_{\lambda,n}\}.
	\end{gather}
	
	These numbers can be computed via the classical 
	representation theory of the symmetric group $S_n$. To
	this end, let $\lambda$ be a Young diagram drawn in English notation.
	Pick up a box $\square$ and denote by $a(\square)$ the \textit{arm-length} of $\square$ defined
	as the number of boxes lying strictly to the right of $\square$.
	Similarly denote by $\ell(\square)$ the \textit{leg-length} the
	number of boxes lying strictly below $\square$.
	We define $d^\lambda$ the hook product of $\lambda$ as 
	\begin{equation}
		d_\lambda\bydef\prod_{\square\in\lambda}(a(\square)+\ell(\square)+1).
	\end{equation}
	
	The quantity $d^\lambda$ has a well known representation theoretic interpretation,
	namely 
	
	\begin{equation}
		d_\lambda = \frac{|\lambda|!}{\chi^\lambda(1)}
	\end{equation}
	where $|\lambda|=\lambda_1+\lambda_2+\dots$
	and $\chi^\lambda(1)$ is the dimension of the irreducible representation of $S_n$ indexed
	by $\lambda$ (see \cite{Stanley}). The following proposition can be found
	in \cite{Stanley} exercise 7.68. For completeness we provide here a proof. 
	
	\begin{proposition}\label{Prop:EnumOriDisco}
		Denote by $p_i$ the $i^{th}$ power sum, and by $s_i$ 
		the $i^{th}$ Schur polynomial. For a partition $\lambda=[\lambda_1\lambda_2\dots]$
		put $p_\lambda=p_{\lambda_1}p_{\lambda_2}\dots$ and
		$s_\lambda=s_{\lambda_1}s_{\lambda_2}\dots$. Then the following equality holds:
		
		\begin{equation}\label{Eq:TorusSchur}
			Q_n^\circ(\mathbf{p}) =\sum_{\lambda\vdash n}d_\lambda s_\lambda = \sum_{\lambda\vdash n}
			h_{\lambda,n}^\circ p_\lambda.
		\end{equation}
	\end{proposition}
	
	In particular, Formula \eqref{Eq:TorusSchur} gives an expansion
	of normalized Schur polynomials in terms of the power sums.
	\begin{proof}
		The proof of \eqref{Eq:TorusSchur} is based on the Frobenius formula\footnote{It is also
			called Burnside formula in the literature.} that we recall now.
		Denote by 
		\begin{equation}
			\mathcal{N}(C_{\lambda_1},\dots,C_{\lambda_k},\omega)=\#\{(\sigma_1,\dots,
            \sigma_k,\omega)\in 
			C_{\lambda_1}\times\dots\times C_{\lambda_k}\times C_\rho\,|\, \sigma_1\dots \sigma_k
            =\omega\}\,,
		\end{equation}
		where $C_{\lambda_i}$'s are arbitrary conjugacy classes in $S_n$ and $C_\rho$ is fixed. Then the Frobenius formula states that
		
		\begin{equation}\label{Eq:FrobForm}
			\mathcal{N}(C_{\lambda_1},\dots,C_{\lambda_k},\omega)= \frac{|C_{\lambda_1}|\dots|C_{\lambda_k}|}{n!}
			\sum_{\mu \vdash n}\frac{\chi^\mu(C_{\lambda_1})\dots\chi^\mu(C_{\lambda_k})\chi^\mu(\omega)}{\chi^\mu(1)^{k-1}}\,,
		\end{equation}
		where $\chi^\mu$ is the character of the irreducible representation
		indexed by $\mu$ of dimension $\chi^\mu(1)$, and  $\chi^\mu(C_{\lambda})$ is its value on the conjugacy class $C_\lambda$.
		
		We refer to \cite{LZ04}, Appendix A, for a proof and further details.
		
		For a given $\sigma\in S_n$ define
		\begin{equation}
			N(\sigma) = \#\{(\alpha,\beta)\in S_n\times S_n\ | \ \alpha\beta\alpha^{-1}\beta^{-1}=\sigma\}.
		\end{equation} 
		The function $N$, the sum of all commutators, is a function class on $S_n$, i.e constant
		on the conjugacy classes, and thus can be expressed in terms of characters of $S_n$.
		Note that  $\alpha$ and $\alpha^{-1}$ are in the same conjugacy class, say $C_\lambda$.
		If $\hat{\alpha}$ is a fixed conjugate of $\alpha^{-1}$, then there are 
		$\frac{n!}{|C_\lambda|}$ elements $\beta\in S_n$ such that  
		$\hat{\alpha} = \beta\alpha^{-1}\beta^{-1}$, giving
		
		\begin{equation}
			N(\sigma) = \sum_{\lambda\vdash n}\frac{n!}{|C_\lambda|}
			\#\{(\alpha,\hat{\alpha})\in C_\lambda\times C_\lambda\ | \ \alpha\hat{\alpha}=\sigma\}.
		\end{equation} 
		Using Formula \eqref{Eq:FrobForm}, we get
		\begin{gather}
			\#\{(\alpha,\hat{\alpha})\in C_\lambda\times C_\lambda\ | \ \alpha\hat{\alpha}=\sigma\}=
			\frac{|C_\lambda|^2}{n!}\sum_{\mu\vdash n}\frac{\chi^\mu(C_\lambda)\chi^\mu(C_\lambda)
				\chi^\mu(\sigma)}{\chi^\mu(1)}\,,
		\end{gather}
		that implies
		\begin{equation}
			N = \sum_{\lambda}|C_\lambda|\sum_{\mu\vdash n}\frac{\chi^\mu(C_\lambda)\chi^\mu(C_\lambda)
				\chi^\mu}{\chi^\mu(1)}.
		\end{equation}
		The group algebra $\Complex[S_n]$ carries
		an inner product given by
		\begin{equation}
			\langle f,g\rangle=\frac{1}{n!}\sum_{\sigma\in S_n}f(\sigma)g(\sigma).
		\end{equation}
		From there we get that 
		\begin{gather}
				\langle N,\chi^\mu\rangle = \frac{1}{\chi^\mu(1)}\sum_\lambda |C_\lambda|
				\chi^\mu(C_\lambda)^2=\frac{n!}{\chi^\mu(1)} = d_\mu
		\end{gather}
		by the orthogonality of the characters.
		
		The Frobenius characteristic map $ch$ from the space of
		class functions $f$ to the space of symmetric polynomials is given by
		\begin{equation}
			ch(f) =\frac{1}{n!}\sum_{\sigma \in S_n}f(\sigma)p_{c(\sigma)}= \sum_{\mu\vdash n} \frac{1}{z_\mu}f(\mu)p_\mu\,,
		\end{equation}
		where $c(\sigma)$ stands for the cycle type of $\sigma\in S_n$, 
		$\mu=[1^{m_1} 2^{m_2} \ldots]$, and $z_\mu = \prod i^{m_i}m_i!$.
		Using the formula for the Schur polynomials
		$$s_\lambda = \sum_\mu\frac{1}{z_\mu}\chi^\lambda(\mu)p_\mu\,,$$
		we finally get the following equality (that is actually true for any class
		function of $S_n$)
		\begin{equation}
			ch(N) = \sum_{\lambda\vdash n} \langle N,\chi^\lambda\rangle s_\lambda
		\end{equation}
		finishing the proof.
	\end{proof}
	
	Consider now the generating function
	
	\begin{equation}\label{Eq:GenFunct}
		\mathcal{Q}(\mathbf{p},t)^\circ :=\sum_n{Q}_n^\circ(\mathbf{p}) t^n= \sum_n\sum_{\lambda\vdash n}
		h_{\lambda,n}^\circ p_\lambda t^n.
	\end{equation}
	Its logarithm 
	
	\begin{equation}\label{Eq:GenFunctConn}
		\mathcal{Q}(\mathbf{p},t):=\log\mathcal{Q}(\mathbf{p},t)^\circ=\sum_n Q_n(\mathbf{p})t^n
	\end{equation}
	enumerates \emph{connected} complex origami. More precisely, the coefficient $$Q_n(\mathbf{p})=\sum_{\lambda\vdash n} h_{\lambda,n} p_\lambda$$
    is a homogeneous polynomial 
    in $p_1,p_2,\ldots$ of degree $n$, where $p_i$ has degree $i$. 
    The number $h_{\lambda, n}$ counts the origami of degree $n$ and ramification profile 
    $\lambda\vdash n$ with weights reciprocal to the orders of their groups of automorphisms.
    
    Proposition \ref{Prop:EnumOriDisco} allows us to compute the polynomials $Q_n$ relatively fast (it took about 40 min. of an average PC time to get $Q_n$ up to $n=30$). The first
    15 polynomials are listed in Appendix B\footnote{The coefficients $h_{\lambda,n}$ with
    $\lambda$ of the form $[2i+1,\, 1^j],\;i=1,2,3,$ and $2i+1+j\leq n\leq 15$ were previously computed in \cite{Nenasheva} by a different method
    and agree with our results.}.
	
	\subsection{Enumeration by degree and genus}
    A simple specialization of \eqref{Eq:GenFunctConn}) provides the enumeration of
    origami by genus and degree. Under the substitution $p_i\mapsto x^{i-1}$ we have
    \begin{equation}
Q_n(\mathbf{p})\big |_{p_i=x^{i-1}}=\sum_{\lambda\vdash n} h_{\lambda,n} \prod_{i=1}^k x^{\lambda_i-1}\,,
        \end{equation}
	where $\lambda=[\lambda_1\ldots\lambda_k]$ is a partition of $n$ of length $k=\ell(\lambda)$.
    Regrouping the terms in the right hand side of the above formula and taking the Riemann-Hurwitz formula into account, we get
    \begin{equation}
        Q_n(\mathbf{p})\big |_{p_i=x^{i-1}}=\sum_{g\geq1,n\geq1}
		H_{g,n}x^{2g-2}\,,
    \end{equation}
	where 
\begin{equation}
    H_{g,n} = \sum_{\substack{\lambda\vdash n\\\ell(\lambda)=n+2-2g}}h_{\lambda,n}
	\end{equation} 
        is the number of connected origami of degree $n$ and genus $g$;
see Tables 1 to 4 at the end of the paper
	for the results up to degree 30 and genus 15.
    
	\subsection{Quasimodularity of generating functions} 
As it is proven in \cite{EO2001} the generating functions $\mathcal{Q}_g(t)=\sum_n H_{g,n} t^n$ are quasimodular for any fixed $g\geq 2$.
%(for the definition of quasimodularity see e.g. \cite{Zagier08}). 
Here we describe a different approach to the quasimodularity of these generating functions. To do it,
we will relate them to the results of \cite{HIL22} using the
unique decomposition of an arbitrary permutation into a product of strictly monotone transpositions.
	
	To be more precise, a sequence of transpositions $\tau_1,\dots\tau_k,\;\tau_i=(a_i,b_i)$,
	is called \emph{strictly monotone} if $a_i<b_i$ and $b_i<b_{i+1}$ for all $i=1,\ldots,k$.
	Introduce the numbers $H^\sim_{g,n}$ by the formula
		\begin{align}\label{Def}
			\qquad H^\sim_{g,n}=&\frac{1}{n!}
            \#\left\{(\sigma,\rho,
		\tau_1,\dots,\tau_{2g-2})\,|\,\right. \\
		&\left.
		\tau_i\in C_{[2,1^{n-2}]},\; (\sigma,\rho)\in S_n\times S_n\ 
		\text{and}\ 
		\sigma\rho\sigma^{-1}\rho^{-1}\tau_1,\dots,\tau_{2g-2}=id  \right\}\,,
		\end{align}
		where the sequence of transpositions $\tau_1,\dots,\tau_{2g-2}$ 
        \begin{itemize}
            \item is strictly monotone, and
            \item $\langle \tau_1,\dots,\tau_{2g-2},\sigma,\rho \rangle$ acts
		transitively on $\{1,\dots,n\}.$
	\end{itemize}
	The numbers $H^\sim_{g,n}$ are known in the literature 
	as (a special case of) connected double strictly monotone Hurwitz numbers.
    
	The following fact was established in \cite{HIL22}:
	
	\begin{theorem}[\cite{HIL22}, Thm 4.4]
		For a fixed genus $g\geq2$ the generating function
		$${\mathcal{Q}}_g^\sim(t)=\sum_{n}H^\sim_{g,n}t^n$$
		is a quasimodular form of mixed weight not larger than $6g-6$.\footnote{What we denote by 
			$H^\sim_{g,n}$ is denoted by $H^{g,d}_{0,0,2g-2}()$ in
			\cite{HIL22}.}
	\end{theorem}

    Quasimodularity of the generating function $\mathcal{Q}_g(t)=\sum_n H_{g,n} t^n$ is then an immediate consequence of
    
	\begin{theorem}\label{Th:HsimisH}
		In the notation as above, we have 
        \begin{equation}
        H^\sim_{g,n}=H_{g,n}.
        \end{equation}
		\end{theorem}
        
	 \begin{proof}
	     The proof is based on
         \begin{lemma}[\cite{BCDG19}, Lemma 2.5]
        \label{Lem:Monot}
		Each permutation in $S_n$
		can be uniquely decomposed into the product of a strictly monotone sequence of
		transpositions. Moreover, if the permutation has cycle type $\lambda$, 
		then the number of transpositions is $n-\ell(\lambda)$, where $\ell(\lambda)$ is the
		length of the partition $\lambda$.
         \end{lemma}
         	
	In combinatorial terms, the enumeration of origami is given by 
	\begin{equation}\label{Eq:Commut}
		n!\cdot H_{g,n}=\# \{(\sigma,\rho)\in S_n\times S_n|\ \sigma\rho\sigma^{-1}\rho^{-1} =\omega\}
	\end{equation}
	  with a permutation $\omega$ in some fixed conjugacy class $C_\lambda$ of $S_n$
	where $\lambda=[\lambda_1,\dots,\lambda_k]$.
	  By lemma \ref{Lem:Monot}, we can further rewrite \eqref{Eq:Commut} as
	\begin{equation}
		n!\cdot H_{g,n}=\#\{\sigma,\rho)\in S_n\times S_n|\ \sigma\rho\sigma^{-1}\rho^{-1} =\tau_1\dots\tau_{n-\ell(\lambda)}\} 
	\end{equation}
        in a \emph{unique} way,
	where $\tau_1,\dots,\tau_{n-\ell(\lambda)}$ is a strictly increasing sequence of transpositions.
	Furthermore, $n-\ell(\lambda)=2g-2$ by the Riemann--Hurwitz formula, so that the right hand side
        in \eqref{Eq:Commut} becomes the definition \eqref{Def} of $H_{g,n}^\sim$, thus completing the proof.
	\end{proof}
	
	\section{Origami and Jack functions}\label{Section:JackFunction}
	Here we observe how the Jack symmetric functions interpolate between the complex and real origamis
    count. 
    
	We introduce the Jack functions following \cite{Mac95}.
	First, define a partial order on the set of partitions given by
	\begin{equation}
		\lambda\leq\mu\, \Leftrightarrow\, |\lambda|=|\mu|\text{ and } 
		\mu_1+\dots+\mu_i\leq\lambda_1+\dots+\lambda_i\text{ for $i\geq1$}
	\end{equation}
	
	Denote by $\Lambda$ the algebra of symmetric functions with rational
	coefficient. For any partition $\mu$ we denote by $m_\mu$ the
	monomial symmetric function, by $p_\mu$ the power sum function and
	by $s_\mu$ the Schur function associated to the partition $\mu$.
	For $\alpha$ an indeterminate let $\Lambda^\alpha := \mathbb{Q}[\alpha]\otimes\Lambda$ 
	the algebra of symmetric function with rational coefficient in $\alpha$.
	We denote by $\langle,\rangle_\alpha$ the $\alpha$-deformation of the Hall scalar product
	defined by
	\begin{equation}
		\langle p_\lambda, p_\mu \rangle = z_\lambda\alpha^{\ell(\lambda)}\delta_{\lambda,\mu}
	\end{equation}
	where $\lambda$ and $\mu$ are partitions, $z_\lambda$ is given by Definition \ref{Def:ZonPoly},
	and $\ell(\lambda)$ is the number of parts of the partition $\lambda$.
	\begin{definition}[\cite{Mac95} Chapter VI.10]
		There exists a unique family
		of symmetric functions $J_\lambda^{(\alpha)}\in \Lambda^\alpha$, called the Jack functions,
		indexed by partitions and satisfying the following properties
		\begin{itemize}
			\item orthogonality: $\langle J^{(\alpha)}_\lambda,J^{(\alpha)}_\mu\rangle = 0$ for $\lambda\neq\mu$,
			\item triangularity: $[m_\mu]J^{(\alpha)}_\lambda = 0$ except if $\mu\leq\lambda$,
			\item normalization: $[m_{1^n}]J^{(\alpha)}_\lambda = n!$ for $\lambda\vdash n$,
		\end{itemize} 
		where $[m_\mu]J^{(\alpha)}_\lambda$ denotes the coefficient by the monomial $m_\mu$ in 
		$J^{(\alpha)}_\lambda$, and  $[1^n]$ is the partition with $n$ parts equal to 1.
	\end{definition}
	
	In particular, for $\alpha=1$ and $\alpha=2$, the Jack polynomials are given by
	\begin{equation}
		J^{(1)}_\lambda=d_\lambda s_\lambda \text{ and } J^{(2)}_\lambda=z_\lambda.
	\end{equation}
	
	It was proved in Section 2 that the expansion
	in terms of power sums of 
	$\sum_n\sum_{\lambda\vdash n} J^{(2)}_\lambda t^n$
	is the generating function $\mathcal{P}^\circ(t,\mathbf{p})$ for (possibly disconnected) real origami.
	Respectively, in Section 3 we have shown that the expansion
	in terms of power sums of  $\sum_n\sum_{\lambda\vdash n} J^{(1)}_\lambda t^n$
	is the generating function $\mathcal{Q}^\circ(t,\mathbf{p})$ for (possibly 
	disconnected) complex origami.
	Taking their logarithm we get the generating functions that enumerate {\it connected} origamis,
    both real and complex. Thus, we arrive at
	
	\begin{proposition}
The	Jack polynomials interpolate between the numbers of real and complex origamis.
	\end{proposition}
    Precisely, it means that the generating functions for the numbers of real and complex origamis admit identical
    expressions in terms of the Jack polynomials that differ only by the value of the parameter $\alpha$.
	
	Let  $b := \alpha-1$ and define the following expansion in terms of
	power sums depending on the parameter $b$:
	\begin{equation}\label{Eq:Interpolation}
		\mathcal{R}(t,\mathbf{p};b):=\log\left(\sum_n\sum_{\lambda\,\vdash n} J^{(b)}_\lambda t^n\right) = 
		\sum_n\sum_{\lambda\,\vdash n}r_\lambda(b)p_\lambda t^n\,;
	\end{equation} in particular, $\mathcal{R}(t,\mathbf{p};0) = \mathcal{Q}(t,\mathbf{p})$ and 
	$\mathcal{R}(t,\mathbf{p};1) = \mathcal{P}(t,\mathbf{p})$.
	Denote by $\mathcal{R}_\mu(t,\mathbf{p};0)$ the generating function for the numbers
of complex connected origami in $\mathcal{H}(\mu_1-1,\ldots,\mu_k-1)$.
	Similarly, denote by $\mathcal{R}_\mu(t,\mathbf{p};1)$ the generating function for the numbers
of connected real origami in $\mathcal{H}(\mu_1-1,\mu_1-1,\ldots,\mu_k-1,\mu_k-1)$.
	(For example, $\mathcal{R}_{[2]}(t,\mathbf{p};1)$ is computed in Theorem \ref{Th:Genfunction}.)

	By \cite{EO2001} $\mathcal{R}_\mu(t,\mathbf{p};0)$  is a quasimodular form
	for any $\mu$. If Conjecture \ref{Conj:QMF} is true, then
	
	\begin{conjecture}
		The generating series $\mathcal{R}_\mu(t,\mathbf{p};b)$ interpolates 
		between the quasimodular forms given by the generating functions for the numbers of connected 
            complex origamis belonging
		to $\mathcal{H}(\mu_1-1,\ldots,\mu_k-1)$
		and the quantum modular forms given by the generating function for the numbers of
		connected real origamis belonging to
		$\mathcal{H}(\mu_1-1,\mu_1-1,\ldots,\mu_k-1,\mu_k-1)$.
	\end{conjecture}
	
	We finish this paper with two open questions.
	The first one has the flavor of the $b-$conjecture from \cite{GJ96a} (see \cite{CD} for recent results).
	
	\begin{question}
	Is it possible to interpret the coefficients $r_\lambda(b)$ of the series $\mathcal{R} 
        (t,\mathbf{p};b)$ in terms of origami count for other values of $b$?
	\end{question}
	
	The second open question is related to integrable hierarchies. Recall that
	\begin{equation}
		\frac{1}{d_\lambda}= d_\lambda^{-1}= s_\lambda(1,0,0,\dots)
	\end{equation}
	where $d_\lambda$ is the hook product and $s_\lambda(1,0,0,\dots)$ is a
	specialization of the Schur function. 
	
	The following facts are well known:
	\begin{proposition}
		The series
		\begin{equation}\label{Eq:KP}
			\tau_{-1}(\mathbf{p}) = \sum_\lambda d_\lambda^{-1} s_\lambda(\mathbf{p})
            %=\exp{(p_1)}
		\end{equation}
		is a tau function of the KP (Kadomtsev--Petviashvili) integrable hierarchy, see \cite{Ok2000},
        while the series
				\begin{equation}\label{Eq:BKP}
		\tau_{0}(\mathbf{p}) =	\sum_\lambda d_\lambda^{0} s_\lambda(\mathbf{p})
		\end{equation}
		is a tau function of the large BKP %(the KP sub-hierarchy of B-type)
        integrable hierarchy, see \cite{KvdL98} and  \cite{NO2017}, Section 5.2.
	\end{proposition}
	
	This raises the following natural question:
	\begin{question}
	Consider the series $\tau_{c}(\mathbf{p}) :=\sum_\lambda d_\lambda^{c} s_\lambda(\mathbf{p})$
	 for a general integer $c$. Is it related to any integrable hierarchy? (For $c=1$ this is
     the generating function for the numbers of complex origami, see Proposition \ref{Prop:EnumOriDisco}.)
	
	\end{question}

     \subsection*{Acknowledgments}
The authors thank Marvin Anas Hahn and Hannah Markwig for the detailed explanations of their results from \cite{HIL22} and \cite{HM24}. Our special thanks are to Valery Gritsenko for drawing our attention to quantum modular forms.
We also extend our gratitude to Maxim Kazarian, Maksim Karev, Sergei Lando, Alexander Orlov and Anton Zorich for stimulating discussions. 
The first author is grateful to the Steklov Mathematical Institute in St. Petersburg for its hospitality during the preparation of this work.
%where a significant portion of the text was written.

	\lstset{basicstyle=\ttfamily,breaklines=true,columns=flexible}
	\pagestyle{empty}
	% \mapleresult
		
	\begin{landscape}
    
		\section*{Appendix A} 
        \noindent
        The first 13 polynomials $P_n$ for real origami count.
		\begin{equation}
			\begin{gathered}
				P_{1} = 
				\begin{minipage}[t]{1 \displaywidth}
					%\raggedright
					\linespread{1.2}
					\selectfont
					\begin{math}
						p_{1}
					\end{math}
				\end{minipage}
			\end{gathered}
		\end{equation}
		% \mapleresult
		\begin{equation}
			\begin{gathered}
				P_{2} = 
				\begin{minipage}[t]{1 \displaywidth}
					%\raggedright
					\linespread{1.2}
					\selectfont
					\begin{math}
						\frac{3 p_{1}^{2}}{2}+p_{2}
					\end{math}
				\end{minipage}
			\end{gathered}
		\end{equation}
		% \mapleresult
		\begin{equation}
			\begin{gathered}
				P_{3} = 
				\begin{minipage}[t]{1 \displaywidth}
					%\raggedright
					\linespread{1.2}
					\selectfont
					\begin{math}
						\frac{4}{3} p_{1}^{3}+3 p_{1} p_{2}+8 p_{3}
					\end{math}
				\end{minipage}
			\end{gathered}
		\end{equation}
		% \mapleresult
		\begin{equation}
			\begin{gathered}
				P_{4} = 
				\begin{minipage}[t]{1 \displaywidth}
					%\raggedright
					\linespread{1.2}
					\selectfont
					\begin{math}
						\frac{7}{4} p_{1}^{4}+7 p_{1}^{2} p_{2}+26 p_{1} p_{3}+\frac{35}{2} p_{2}^{2}+36 p_{4}
					\end{math}
				\end{minipage}
			\end{gathered}
		\end{equation}
		% \mapleresult
		\begin{equation}
			\begin{gathered}
				P_{5} = 
				\begin{minipage}[t]{1 \displaywidth}
					%\raggedright
					\linespread{1.2}
					\selectfont
					\begin{math}
						10 p_{1}^{3} p_{2}+74 p_{1}^{2} p_{3}+63 p_{1} p_{2}^{2}+180 p_{1} p_{4}+128 p_{2} p_{3}+360 p_{5}+\frac{6}{5} p_{1}^{5}
					\end{math}
				\end{minipage}
			\end{gathered}
		\end{equation}
		% \mapleresult
		\begin{equation}
			\begin{gathered}
				P_{6} = 
				\begin{minipage}[t]{1 \displaywidth}
					%\raggedright
					\linespread{1.2}
					\selectfont
					\begin{math}
						790 p_{1} p_{2} p_{3}+19 p_{1}^{4} p_{2}+132 p_{1}^{3} p_{3}+\frac{421}{2} p_{1}^{2} p_{2}^{2}+510 p_{1}^{2} p_{4}+2072 p_{1} p_{5}+1256 p_{2} p_{4}+3364 p_{6}+2 p_{1}^{6}+\frac{292}{3} p_{2}^{3}+748 p_{3}^{2}
					\end{math}
				\end{minipage}
			\end{gathered}
		\end{equation}
		% \mapleresult
		\begin{equation}
			\begin{gathered}
				P_{7} = 
				\begin{minipage}[t]{1 \displaywidth}
					%\raggedright
					\linespread{1.2}
					\selectfont
					\begin{math}
						2616 p_{1}^{2} p_{2} p_{3}+9368 p_{1} p_{2} p_{4}+21 p_{1}^{5} p_{2}+250 p_{1}^{4} p_{3}+425 p_{1}^{3} p_{2}^{2}+1232 p_{1}^{3} p_{4}+667 p_{1} p_{2}^{3}+7212 p_{1}^{2} p_{5}+4320 p_{1} p_{3}^{2}+3228 p_{2}^{2} p_{3}+23272 p_{1} p_{6}+14268 p_{2} p_{5}+12140 p_{3} p_{4}+43008 p_{7}+\frac{8}{7} p_{1}^{7}
					\end{math}
				\end{minipage}
			\end{gathered}
		\end{equation}
		% \mapleresult
		\begin{equation}
			\begin{gathered}
				P_{8} = 
				\begin{minipage}[t]{1 \displaywidth}
					%\raggedright
					\linespread{1.2}
					\selectfont
					\begin{math}7152 p_{1}^{3} p_{2} p_{3}+37308 p_{1}^{2} p_{2} p_{4}+25116 p_{1} p_{2}^{2} p_{3}+111148 p_{1} p_{2} p_{5}+98212 p_{1} p_{3} p_{4}+35 p_{1}^{6} p_{2}+390 p_{1}^{5} p_{3}+\frac{1845}{2} p_{1}^{4} p_{2}^{2}+2404 p_{1}^{4} p_{4}+2773 p_{1}^{2} p_{2}^{3}+18452 p_{1}^{3} p_{5}+18018 p_{1}^{2} p_{3}^{2}+91512 p_{1}^{2} p_{6}+37360 p_{2}^{2} p_{4}+31480 p_{2} p_{3}^{2}+338616 p_{1} p_{7}+200372 p_{2} p_{6}+159096 p_{3} p_{5}+595872 p_{8}+\frac{15}{8} p_{1}^{8}+\frac{7855}{4} p_{2}^{4}+81300 p_{4}^{2}
					\end{math}
				\end{minipage}
			\end{gathered}
		\end{equation}
		% \mapleresult
		\begin{equation}
			\begin{gathered}
				P_{9} = 
				\begin{minipage}[t]{1 \displaywidth}
					%\raggedright
					\linespread{1.2}
					\selectfont
					\begin{math}
						15480 p_{1}^{4} p_{2} p_{3}+109640 p_{1}^{3} p_{2} p_{4}+115154 p_{1}^{2} p_{2}^{2} p_{3}+505236 p_{1}^{2} p_{2} p_{5}+427408 p_{1}^{2} p_{3} p_{4}+322216 p_{1} p_{2}^{2} p_{4}+289148 p_{1} p_{2} p_{3}^{2}+1792892 p_{1} p_{2} p_{6}+1443088 p_{1} p_{3} p_{5}+913052 p_{2} p_{3} p_{4}+39 p_{1}^{7} p_{2}+572 p_{1}^{6} p_{3}+1523 p_{1}^{5} p_{2}^{2}+4336 p_{1}^{5} p_{4}+7985 p_{1}^{3} p_{2}^{3}+42600 p_{1}^{4} p_{5}+50336 p_{1}^{3} p_{3}^{2}+14263 p_{1} p_{2}^{4}+274412 p_{1}^{3} p_{6}+73072 p_{2}^{3} p_{3}+1514640 p_{1}^{2} p_{7}+689124 p_{1} p_{4}^{2}+552924 p_{2}^{2} p_{5}+5316368 p_{1} p_{8}+3079760 p_{2} p_{7}+2399024 p_{3} p_{6}+2176784 p_{4} p_{5}+9732480 p_{9}+\frac{13}{9} p_{1}^{9}+\frac{415568}{3} p_{3}^{3}
					\end{math}
				\end{minipage}
			\end{gathered}
		\end{equation}
		% \mapleresult
		\begin{equation}
			\begin{gathered}
				P_{10} = 
				\begin{minipage}[t]{1 \displaywidth}
					%\raggedright
					\linespread{1.2}
					\selectfont
					\begin{math}
						31880 p_{1}^{5} p_{2} p_{3}+276200 p_{1}^{4} p_{2} p_{4}+381852 p_{1}^{3} p_{2}^{2} p_{3}+1667396 p_{1}^{3} p_{2} p_{5}+1438732 p_{1}^{3} p_{3} p_{4}+1612012 p_{1}^{2} p_{2}^{2} p_{4}+1435712 p_{1}^{2} p_{2} p_{3}^{2}+736416 p_{1} p_{2}^{3} p_{3}+8847740 p_{1}^{2} p_{2} p_{6}+7208984 p_{1}^{2} p_{3} p_{5}+5485904 p_{1} p_{2}^{2} p_{5}+30502912 p_{1} p_{2} p_{7}+23714864 p_{1} p_{3} p_{6}+21771640 p_{1} p_{4} p_{5}+14693000 p_{2} p_{3} p_{5}+175455360 p_{10}+56 p_{1}^{8} p_{2}+854 p_{1}^{7} p_{3}+2713 p_{1}^{6} p_{2}^{2}+7260 p_{1}^{6} p_{4}+20654 p_{1}^{4} p_{2}^{3}+82140 p_{1}^{5} p_{5}+128886 p_{1}^{4} p_{3}^{2}+\frac{149129}{2} p_{1}^{2} p_{2}^{4}+674208 p_{1}^{4} p_{6}+4998640 p_{1}^{3} p_{7}+3460860 p_{1}^{2} p_{4}^{2}+1322144 p_{1} p_{3}^{3}+1122384 p_{2}^{3} p_{4}+1581772 p_{2}^{2} p_{3}^{2}+26393184 p_{1}^{2} p_{8}+8987972 p_{2}^{2} p_{6}+6735424 p_{2} p_{4}^{2}+6003272 p_{3}^{2} p_{4}+96544960 p_{1} p_{9}+54907152 p_{2} p_{8}+41882080 p_{3} p_{7}+36556576 p_{4} p_{6}+9129800 p_{1} p_{2} p_{3} p_{4}+\frac{9}{5} p_{1}^{10}+\frac{134316}{5} p_{2}^{5}+17986368 p_{5}^{2}
					\end{math}
				\end{minipage}
			\end{gathered}
		\end{equation}
		% \mapleresult
		\begin{equation}
			\begin{gathered}
				P_{11} = 
				\begin{minipage}[t]{1 \displaywidth}
					%\raggedright
					\linespread{1.2}
					\selectfont
					\begin{math}
						55 p_{1}^{9} p_{2}+1140 p_{1}^{8} p_{3}+3870 p_{1}^{7} p_{2}^{2}+11384 p_{1}^{7} p_{4}+43404 p_{1}^{5} p_{2}^{3}+153484 p_{1}^{6} p_{5}+272328 p_{1}^{5} p_{3}^{2}+264251 p_{1}^{3} p_{2}^{4}+1491220 p_{1}^{5} p_{6}+300805 p_{1} p_{2}^{5}+13705008 p_{1}^{4} p_{7}+12392112 p_{1}^{3} p_{4}^{2}+7385264 p_{1}^{2} p_{3}^{3}+2122172 p_{2}^{4} p_{3}+96366464 p_{1}^{3} p_{8}+19973912 p_{2}^{3} p_{5}+14883776 p_{2} p_{3}^{3}+528511152 p_{1}^{2} p_{9}+193616576 p_{1} p_{5}^{2}+173402912 p_{2}^{2} p_{7}+108350608 p_{3}^{2} p_{5}+101430112 p_{3} p_{4}^{2}+1918857600 p_{1} p_{10}+1076045456 p_{2} p_{9}+807949616 p_{3} p_{8}+692682192 p_{4} p_{7}+648080528 p_{5} p_{6}+3538971648 p_{11}+12584624 p_{1} p_{2}^{3} p_{4}+16850520 p_{1} p_{2}^{2} p_{3}^{2}+167234832 p_{1}^{2} p_{2} p_{7}+130549184 p_{1}^{2} p_{3} p_{6}+118145736 p_{1}^{2} p_{4} p_{5}+98793352 p_{1} p_{2}^{2} p_{6}+74388360 p_{1} p_{2} p_{4}^{2}+66080032 p_{1} p_{3}^{2} p_{4}+50078040 p_{2}^{2} p_{3} p_{4}+600663776 p_{1} p_{2} p_{8}+458233888 p_{1} p_{3} p_{7}+402046336 p_{1} p_{4} p_{6}+270568528 p_{2} p_{3} p_{6}+243473664 p_{2} p_{4} p_{5}+56242 p_{1}^{6} p_{2} p_{3}+600920 p_{1}^{5} p_{2} p_{4}+1035758 p_{1}^{4} p_{2}^{2} p_{3}+4587924 p_{1}^{4} p_{2} p_{5}+3874532 p_{1}^{4} p_{3} p_{4}+5861024 p_{1}^{3} p_{2}^{2} p_{4}+5235648 p_{1}^{3} p_{2} p_{3}^{2}+3974286 p_{1}^{2} p_{2}^{3} p_{3}+32531652 p_{1}^{3} p_{2} p_{6}+26243392 p_{1}^{3} p_{3} p_{5}+29801064 p_{1}^{2} p_{2}^{2} p_{5}+49884592 p_{1}^{2} p_{2} p_{3} p_{4}+158652640 p_{1} p_{2} p_{3} p_{5}+\frac{12}{11} p_{1}^{11}
					\end{math}
				\end{minipage}
			\end{gathered}
		\end{equation}
		% \mapleresult
        
		\begin{equation}
			\begin{gathered}
				P_{12} = 
				\begin{minipage}[t]{1 \displaywidth}
					%\raggedright
					\linespread{1.2}
					\selectfont
					\begin{math}
						91 p_{1}^{10} p_{2}+1540 p_{1}^{9} p_{3}+\frac{12915}{2} p_{1}^{8} p_{2}^{2}+16998 p_{1}^{8} p_{4}+\frac{273205}{3} p_{1}^{6} p_{2}^{3}+257756 p_{1}^{7} p_{5}+555712 p_{1}^{6} p_{3}^{2}+\frac{3223801}{4} p_{1}^{4} p_{2}^{4}+2943628 p_{1}^{6} p_{6}+1816638 p_{1}^{2} p_{2}^{5}+32613208 p_{1}^{5} p_{7}+37407906 p_{1}^{4} p_{4}^{2}+\frac{88100192}{3} p_{1}^{3} p_{3}^{3}+287566896 p_{1}^{4} p_{8}+38271412 p_{2}^{4} p_{4}+66631804 p_{2}^{3} p_{3}^{2}+2104344944 p_{1}^{3} p_{9}+1162388176 p_{1}^{2} p_{5}^{2}+407596652 p_{2}^{3} p_{6}+462537480 p_{2}^{2} p_{4}^{2}+11464289616 p_{1}^{2} p_{10}+3642334416 p_{2}^{2} p_{8}+2331019824 p_{2} p_{5}^{2}+2175531184 p_{3}^{2} p_{6}+42258213120 p_{1} p_{11}+23426252496 p_{2} p_{10}+17367281152 p_{3} p_{9}+14652355776 p_{4} p_{8}+13389488640 p_{5} p_{7}+78107489280 p_{12}+2404196128 p_{1}^{2} p_{4} p_{6}+2066035248 p_{1} p_{2}^{2} p_{7}+1294670016 p_{1} p_{3}^{2} p_{5}+1207490320 p_{1} p_{3} p_{4}^{2}+972100840 p_{2}^{2} p_{3} p_{5}+815030648 p_{2} p_{3}^{2} p_{4}+12846295888 p_{1} p_{2} p_{9}+9642384528 p_{1} p_{3} p_{8}+8261151696 p_{1} p_{4} p_{7}+7758064752 p_{1} p_{5} p_{6}+5551682880 p_{2} p_{3} p_{7}+4888199952 p_{2} p_{4} p_{6}+3887608864 p_{3} p_{4} p_{5}+445351672 p_{1}^{2} p_{2} p_{4}^{2}+394132080 p_{1}^{2} p_{3}^{2} p_{4}+237898672 p_{1} p_{2}^{3} p_{5}+178529760 p_{1} p_{2} p_{3}^{3}+3586570032 p_{1}^{2} p_{2} p_{8}+2733056752 p_{1}^{2} p_{3} p_{7}+519616368 p_{1}^{3} p_{3} p_{6}+472997216 p_{1}^{3} p_{4} p_{5}+590228172 p_{1}^{2} p_{2}^{2} p_{6}+78639768 p_{1}^{4} p_{3} p_{5}+118718688 p_{1}^{3} p_{2}^{2} p_{5}+75443516 p_{1}^{2} p_{2}^{3} p_{4}+101222238 p_{1}^{2} p_{2}^{2} p_{3}^{2}+25292562 p_{1} p_{2}^{4} p_{3}+664565320 p_{1}^{3} p_{2} p_{7}+96810356 p_{1}^{4} p_{2} p_{6}+99792 p_{1}^{7} p_{2} p_{3}+1200424 p_{1}^{6} p_{2} p_{4}+2515692 p_{1}^{5} p_{2}^{2} p_{3}+10940516 p_{1}^{5} p_{2} p_{5}+9358988 p_{1}^{5} p_{3} p_{4}+17593430 p_{1}^{4} p_{2}^{2} p_{4}+15764100 p_{1}^{4} p_{2} p_{3}^{2}+15975672 p_{1}^{3} p_{2}^{3} p_{3}+199191488 p_{1}^{3} p_{2} p_{3} p_{4}+949890208 p_{1}^{2} p_{2} p_{3} p_{5}+598894168 p_{1} p_{2}^{2} p_{3} p_{4}+3236864904 p_{1} p_{2} p_{3} p_{6}+2909892120 p_{1} p_{2} p_{4} p_{5}+\frac{7}{3} p_{1}^{12}+\frac{2052388}{3} p_{2}^{6}+61638120 p_{3}^{4}+608665728 p_{4}^{3}+6554232744 p_{6}^{2}
					\end{math}
				\end{minipage}
			\end{gathered}
		\end{equation}
        \clearpage
		% \mapleresult
		\begin{equation}
			\begin{gathered}
				P_{13} = 
				\begin{minipage}[t]{1 \displaywidth}
					%\raggedright
					\linespread{1.2}
					\selectfont
					\begin{math}
						643859068 p_{1}^{4} p_{2} p_{3} p_{4}+4093431816 p_{1}^{3} p_{2} p_{3} p_{5}+3874065256 p_{1}^{2} p_{2}^{2} p_{3} p_{4}+20914479904 p_{1}^{2} p_{2} p_{3} p_{6}+18842052576 p_{1}^{2} p_{2} p_{4} p_{5}+12645149408 p_{1} p_{2}^{2} p_{3} p_{5}+10553645728 p_{1} p_{2} p_{3}^{2} p_{4}+71992619712 p_{1} p_{2} p_{3} p_{7}+63054973696 p_{1} p_{2} p_{4} p_{6}+50445809472 p_{1} p_{3} p_{4} p_{5}+78 p_{1}^{11} p_{2}+1958 p_{1}^{10} p_{3}+8305 p_{1}^{9} p_{2}^{2}+24524 p_{1}^{9} p_{4}+160354 p_{1}^{7} p_{2}^{3}+423924 p_{1}^{8} p_{5}+1000824 p_{1}^{7} p_{3}^{2}+2044805 p_{1}^{5} p_{2}^{4}+5497112 p_{1}^{7} p_{6}+7777788 p_{1}^{3} p_{2}^{5}+70668936 p_{1}^{6} p_{7}+95736936 p_{1}^{5} p_{4}^{2}+95413672 p_{1}^{4} p_{3}^{3}+8293615 p_{1} p_{2}^{6}+745897520 p_{1}^{5} p_{8}+65634668 p_{2}^{5} p_{3}+6818573984 p_{1}^{4} p_{9}+4991110064 p_{1}^{3} p_{5}^{2}+783593872 p_{1} p_{3}^{4}+796611148 p_{2}^{4} p_{5}+1180544128 p_{2}^{2} p_{3}^{3}+49520174800 p_{1}^{3} p_{10}+7879041744 p_{1} p_{4}^{3}+9058317248 p_{2}^{3} p_{7}+4712743184 p_{3}^{3} p_{4}+273753243648 p_{1}^{2} p_{11}+84485536624 p_{1} p_{6}^{2}+84858380624 p_{2}^{2} p_{9}+48592010112 p_{3}^{2} p_{7}+40792326592 p_{3} p_{5}^{2}+38243343808 p_{4}^{2} p_{5}+1011362247168 p_{1} p_{12}+555386417280 p_{2} p_{11}+407551789696 p_{3} p_{10}+339652019200 p_{4} p_{9}+305715407616 p_{5} p_{8}+291371529472 p_{6} p_{7}+189776794816 p_{1} p_{4} p_{8}+173644723456 p_{1} p_{5} p_{7}+127364328048 p_{2} p_{3} p_{8}+109256967056 p_{2} p_{4} p_{7}+101982926496 p_{2} p_{5} p_{6}+85032513488 p_{3} p_{4} p_{6}+155782 p_{1}^{8} p_{2} p_{3}+2214512 p_{1}^{7} p_{2} p_{4}+5344350 p_{1}^{6} p_{2}^{2} p_{3}+23698592 p_{1}^{6} p_{2} p_{5}+19983660 p_{1}^{6} p_{3} p_{4}+45421680 p_{1}^{5} p_{2}^{2} p_{4}+40577388 p_{1}^{5} p_{2} p_{3}^{2}+51305366 p_{1}^{4} p_{2}^{3} p_{3}+251938712 p_{1}^{5} p_{2} p_{6}+203127576 p_{1}^{5} p_{3} p_{5}+384416636 p_{1}^{4} p_{2}^{2} p_{5}+324867016 p_{1}^{3} p_{2}^{3} p_{4}+434306912 p_{1}^{3} p_{2}^{2} p_{3}^{2}+164214258 p_{1}^{2} p_{2}^{4} p_{3}+2158625288 p_{1}^{4} p_{2} p_{7}+1685821112 p_{1}^{4} p_{3} p_{6}+1523696152 p_{1}^{4} p_{4} p_{5}+2549409104 p_{1}^{3} p_{2}^{2} p_{6}+1920972056 p_{1}^{3} p_{2} p_{4}^{2}+1706890192 p_{1}^{3} p_{3}^{2} p_{4}+1545765412 p_{1}^{2} p_{2}^{3} p_{5}+1149514008 p_{1}^{2} p_{2} p_{3}^{3}+487600060 p_{1} p_{2}^{4} p_{4}+867439976 p_{1} p_{2}^{3} p_{3}^{2}+15497789280 p_{1}^{3} p_{2} p_{8}+11829753456 p_{1}^{3} p_{3} p_{7}+10378483712 p_{1}^{3} p_{4} p_{6}+13411499296 p_{1}^{2} p_{2}^{2} p_{7}+8373969184 p_{1}^{2} p_{3}^{2} p_{5}+7847449728 p_{1}^{2} p_{3} p_{4}^{2}+5260464648 p_{1} p_{2}^{3} p_{6}+5933482352 p_{1} p_{2}^{2} p_{4}^{2}+2655872524 p_{2}^{3} p_{3} p_{4}+83233125264 p_{1}^{2} p_{2} p_{9}+62485128224 p_{1}^{2} p_{3} p_{8}+53613891552 p_{1}^{2} p_{4} p_{7}+50122945456 p_{1}^{2} p_{5} p_{6}+47173787152 p_{1} p_{2}^{2} p_{8}+30266786896 p_{1} p_{2} p_{5}^{2}+28040715136 p_{1} p_{3}^{2} p_{6}+21147875680 p_{2}^{2} p_{3} p_{6}+19049391216 p_{2}^{2} p_{4} p_{5}+16931248224 p_{2} p_{3}^{2} p_{5}+15881677056 p_{2} p_{3} p_{4}^{2}+303335034544 p_{1} p_{2} p_{10}+224865824832 p_{1} p_{3} p_{9}+\frac{14}{13} p_{1}^{13}+1883040929280 p_{13}
					\end{math}
				\end{minipage}
			\end{gathered}
		\end{equation}
	\end{landscape}
	\clearpage
    
	\begin{landscape} 
    \section*{Appendix B}
	\noindent	
        The first 15 polynomials $Q_n$ for complex origami count.	
		\begin{equation}
			Q_{1} = 
			\begin{minipage}[t]{1 \displaywidth}
				\begin{math}
					p_{1}
				\end{math}
			\end{minipage}
		\end{equation}
		\begin{equation}
			Q_{2} = 
			\begin{minipage}[t]{1 \displaywidth}
				\begin{math}
					\frac{3 p_{1}^{2}}{2}
				\end{math}
			\end{minipage}
		\end{equation}		
		\begin{equation}
			Q_{3} = 
			\begin{minipage}[t]{1 \displaywidth}
				\begin{math}
					\frac{4 p_{1}^{3}}{3}+3 p_{3}
				\end{math}
			\end{minipage}
		\end{equation}	
		\begin{equation}
			Q_{4} = 
			\begin{minipage}[t]{1 \displaywidth}
				\begin{math}
					\frac{7}{4} p_{1}^{4}+9 p_{1} p_{3}+7 p_{2}^{2}
				\end{math}
			\end{minipage}
		\end{equation}	
		\begin{equation}
			Q_{5} = 
			\begin{minipage}[t]{1 \displaywidth}
				\begin{math}
					\frac{6}{5} p_{1}^{5}+27 p_{1}^{2} p_{3}+24 p_{1} p_{2}^{2}+40 p_{5}
				\end{math}
			\end{minipage}
		\end{equation}
		\begin{equation}
			Q_{6} = 
			\begin{minipage}[t]{1 \displaywidth}
				\begin{math}
					2 p_{1}^{6}+45 p_{1}^{3} p_{3}+80 p_{1}^{2} p_{2}^{2}+225 p_{1} p_{5}+128 p_{2} p_{4}+\frac{189}{2} p_{3}^{2}
				\end{math}
			\end{minipage}
		\end{equation}
		\begin{equation}
			Q_{7} = 
			\begin{minipage}[t]{1 \displaywidth}
				\begin{math}
					1260 p_{7}+\frac{8}{7} p_{1}^{7}+1024 p_{1} p_{2} p_{4}+90 p_{1}^{4} p_{3}+160 p_{1}^{3} p_{2}^{2}+775 p_{1}^{2} p_{5}+486 p_{1} p_{3}^{2}+360 p_{2}^{2} p_{3}
				\end{math}
			\end{minipage}
		\end{equation}
		\begin{equation}
			Q_{8} = 
			\begin{minipage}[t]{1 \displaywidth}
				\begin{math}
					\frac{495}{2} p_{2}^{4}+4032 p_{1}^{2} p_{2} p_{4}+2700 p_{1} p_{2}^{2} p_{3}+\frac{15}{8} p_{1}^{8}+2584 p_{4}^{2}+9800 p_{1} p_{7}+5832 p_{2} p_{6}+4500 p_{3} p_{5}+135 p_{1}^{5} p_{3}+345 p_{1}^{4} p_{2}^{2}+1925 p_{1}^{3} p_{5}+\frac{4131}{2} p_{1}^{2} p_{3}^{2}
				\end{math}
			\end{minipage}
		\end{equation}
		% \mapleresult
		%\begin{dmath*}
		\begin{equation}
			\begin{gathered}
				Q_{9} = 
				\begin{minipage}[t]{1 \displaywidth}
					%\raggedright
					\linespread{1.2}
					\selectfont
					\begin{math}
						72576 p_{9}+11776 p_{1}^{3} p_{2} p_{4}+12780 p_{1}^{2} p_{2}^{2} p_{3}+51840 p_{1} p_{2} p_{6}+41625 p_{1} p_{3} p_{5}+26496 p_{2} p_{3} p_{4}+\frac{13}{9} p_{1}^{9}+4232 p_{3}^{3}+201 p_{1}^{6} p_{3}+568 p_{1}^{5} p_{2}^{2}+4550 p_{1}^{4} p_{5}+5544 p_{1}^{3} p_{3}^{2}+1560 p_{1} p_{2}^{4}+43512 p_{1}^{2} p_{7}+20608 p_{1} p_{4}^{2}+16500 p_{2}^{2} p_{5}
					\end{math}
				\end{minipage}
			\end{gathered}
		\end{equation}
		%\end{dmath*}
		% \mapleresult
		%\begin{dmath*}
		\begin{equation}
			\begin{gathered}
				Q_{10} = 
				\begin{minipage}[t]{1 \displaywidth}
					%\raggedright
					\linespread{1.2}
					\selectfont
					\begin{math}
						29568 p_{1}^{4} p_{2} p_{4}+41580 p_{1}^{3} p_{2}^{2} p_{3}+252720 p_{1}^{2} p_{2} p_{6}+208125 p_{1}^{2} p_{3} p_{5}+161500 p_{1} p_{2}^{2} p_{5}+\frac{9}{5} p_{1}^{10}+138735 p_{5}^{2}+266112 p_{1} p_{2} p_{3} p_{4}+297 p_{1}^{7} p_{3}+1008 p_{1}^{6} p_{2}^{2}+8470 p_{1}^{5} p_{5}+\frac{28539}{2} p_{1}^{4} p_{3}^{2}+8340 p_{1}^{2} p_{2}^{4}+142835 p_{1}^{3} p_{7}+102976 p_{1}^{2} p_{4}^{2}+37224 p_{1} p_{3}^{3}+31488 p_{2}^{3} p_{4}+48384 p_{2}^{2} p_{3}^{2}+714420 p_{1} p_{9}+409600 p_{2} p_{8}+312669 p_{3} p_{7}+269568 p_{4} p_{6}
					\end{math}
				\end{minipage}
			\end{gathered}
		\end{equation}
		%\end{dmath*}
		% \mapleresult
		\begin{equation}
			\begin{gathered}
				Q_{11} = 
				\begin{minipage}[t]{1 \displaywidth}
					%\raggedright
					\linespread{1.2}
					\selectfont
					\begin{math}
						64512 p_{1}^{5} p_{2} p_{4}+113400 p_{1}^{4} p_{2}^{2} p_{3}+933120 p_{1}^{3} p_{2} p_{6}+756000 p_{1}^{3} p_{3} p_{5}+865500 p_{1}^{2} p_{2}^{2} p_{5}+368640 p_{1} p_{2}^{3} p_{4}+495720 p_{1} p_{2}^{2} p_{3}^{2}+4456448 p_{1} p_{2} p_{8}+3395259 p_{1} p_{3} p_{7}+2985984 p_{1} p_{4} p_{6}+2029536 p_{2} p_{3} p_{6}+6652800 p_{11}+1820800 p_{2} p_{4} p_{5}+\frac{12}{11} p_{1}^{11}+1451520 p_{1}^{2} p_{2} p_{3} p_{4}+405 p_{1}^{8} p_{3}+1440 p_{1}^{7} p_{2}^{2}+16125 p_{1}^{6} p_{5}+29646 p_{1}^{5} p_{3}^{2}+29280 p_{1}^{3} p_{2}^{4}+389305 p_{1}^{4} p_{7}+360960 p_{1}^{3} p_{4}^{2}+214110 p_{1}^{2} p_{3}^{3}+62640 p_{2}^{4} p_{3}+3892860 p_{1}^{2} p_{9}+1448425 p_{1} p_{5}^{2}+1296540 p_{2}^{2} p_{7}+812700 p_{3}^{2} p_{5}+760320 p_{3} p_{4}^{2}
					\end{math}
				\end{minipage}
			\end{gathered}
		\end{equation}
		% \mapleresult
		%\begin{dmath*}
		\begin{equation}
			\begin{gathered}
				Q_{12} = 
				\begin{minipage}[t]{1 \displaywidth}
					%\raggedright
					\linespread{1.2}
					\selectfont
					\begin{math}
						\frac{67420}{3} p_{2}^{6}+\frac{1945689}{4} p_{3}^{4}+127232 p_{1}^{6} p_{2} p_{4}+273420 p_{1}^{5} p_{2}^{2} p_{3}+2757240 p_{1}^{4} p_{2} p_{6}+2256750 p_{1}^{4} p_{3} p_{5}+3429000 p_{1}^{3} p_{2}^{2} p_{5}+2202880 p_{1}^{2} p_{2}^{3} p_{4}+2971620 p_{1}^{2} p_{2}^{2} p_{3}^{2}+735480 p_{1} p_{2}^{4} p_{3}+26517504 p_{1}^{2} p_{2} p_{8}+20125917 p_{1}^{2} p_{3} p_{7}+17770752 p_{1}^{2} p_{4} p_{6}+15324260 p_{1} p_{2}^{2} p_{7}+9653850 p_{1} p_{3}^{2} p_{5}+8945280 p_{1} p_{3} p_{4}^{2}+7222500 p_{2}^{2} p_{3} p_{5}+6092928 p_{2} p_{3}^{2} p_{4}+\frac{7}{3} p_{1}^{12}+12477456 p_{6}^{2}+5765760 p_{1}^{3} p_{2} p_{3} p_{4}+24245568 p_{1} p_{2} p_{3} p_{6}+21654400 p_{1} p_{2} p_{4} p_{5}+32722380 p_{3} p_{9}+27594752 p_{4} p_{8}+25112500 p_{5} p_{7}+525 p_{1}^{9} p_{3}+2380 p_{1}^{8} p_{2}^{2}+26600 p_{1}^{7} p_{5}+\frac{121149}{2} p_{1}^{6} p_{3}^{2}+88740 p_{1}^{4} p_{2}^{4}+921739 p_{1}^{5} p_{7}+1094400 p_{1}^{4} p_{4}^{2}+841740 p_{1}^{3} p_{3}^{3}+15445080 p_{1}^{3} p_{9}+\frac{17425525}{2} p_{1}^{2} p_{5}^{2}+3066120 p_{2}^{3} p_{6}+3537536 p_{2}^{2} p_{4}^{2}+79035264 p_{1} p_{11}+44100000 p_{2} p_{10}
					\end{math}
				\end{minipage}
			\end{gathered}
		\end{equation}
		%\end{dmath*}
		\begin{equation}
			\begin{gathered}
				Q_{13} = 
				\begin{minipage}[t]{1 \displaywidth}
					%\raggedright
					\linespread{1.2}
					\selectfont
					\begin{math}
						889574400 p_{13}+236544 p_{1}^{7} p_{2} p_{4}+582120 p_{1}^{6} p_{2}^{2} p_{3}+7185024 p_{1}^{5} p_{2} p_{6}+5816250 p_{1}^{5} p_{3} p_{5}+11089500 p_{1}^{4} p_{2}^{2} p_{5}+9461760 p_{1}^{3} p_{2}^{3} p_{4}+12668400 p_{1}^{3} p_{2}^{2} p_{3}^{2}+4817880 p_{1}^{2} p_{2}^{4} p_{3}+114163712 p_{1}^{3} p_{2} p_{8}+87106761 p_{1}^{3} p_{3} p_{7}+76640256 p_{1}^{3} p_{4} p_{6}+99496460 p_{1}^{2} p_{2}^{2} p_{7}+62185050 p_{1}^{2} p_{3}^{2} p_{5}+58216320 p_{1}^{2} p_{3} p_{4}^{2}+39268800 p_{1} p_{2}^{3} p_{6}+44394496 p_{1} p_{2}^{2} p_{4}^{2}+19975680 p_{2}^{3} p_{3} p_{4}+568800000 p_{1} p_{2} p_{10}+421694424 p_{1} p_{3} p_{9}+355827712 p_{1} p_{4} p_{8}+325752000 p_{1} p_{5} p_{7}+240205824 p_{2} p_{3} p_{8}+206260992 p_{2} p_{4} p_{7}+192456000 p_{2} p_{5} p_{6}+160621056 p_{3} p_{4} p_{6}+\frac{14}{13} p_{1}^{13}+18627840 p_{1}^{4} p_{2} p_{3} p_{4}+155434464 p_{1}^{2} p_{2} p_{3} p_{6}+140067200 p_{1}^{2} p_{2} p_{4} p_{5}+94446000 p_{1} p_{2}^{2} p_{3} p_{5}+78928128 p_{1} p_{2} p_{3}^{2} p_{4}+693 p_{1}^{10} p_{3}+3080 p_{1}^{9} p_{2}^{2}+44275 p_{1}^{8} p_{5}+108108 p_{1}^{7} p_{3}^{2}+225360 p_{1}^{5} p_{2}^{4}+1993222 p_{1}^{6} p_{7}+2762496 p_{1}^{5} p_{4}^{2}+2752695 p_{1}^{4} p_{3}^{3}+245920 p_{1} p_{2}^{6}+49878180 p_{1}^{4} p_{9}+36985775 p_{1}^{3} p_{5}^{2}+5885586 p_{1} p_{3}^{4}+6001000 p_{2}^{4} p_{5}+8881200 p_{2}^{2} p_{3}^{3}+510314112 p_{1}^{2} p_{11}+158827392 p_{1} p_{6}^{2}+160065072 p_{2}^{2} p_{9}+91846629 p_{3}^{2} p_{7}+76957875 p_{3} p_{5}^{2}+72176000 p_{4}^{2} p_{5}
					\end{math}
				\end{minipage}
			\end{gathered}
		\end{equation}
        \begin{equation}
			\begin{gathered}
				Q_{14} = 
				\begin{minipage}[t]{1 \displaywidth}
					%\raggedright
					\linespread{1.2}
					\selectfont
					\begin{math}
                    1668429476 p_{7}^{2}+3959000000 p_{1}^{2} p_{2} p_{10}+2934227916 p_{1}^{2} p_{3} p_{9}+2473000960 p_{1}^{2} p_{4} p_{8}+2267989500 p_{1}^{2} p_{5} p_{7}+2225062548 p_{1} p_{2}^{2} p_{9}+1277745462 p_{1} p_{3}^{2} p_{7}+1066416750 p_{1} p_{3} p_{5}^{2}+1001712000 p_{1} p_{4}^{2} p_{5}+969478524 p_{2}^{2} p_{3} p_{7}+840927744 p_{2}^{2} p_{4} p_{6}+748898784 p_{2} p_{3}^{2} p_{6}+412416 p_{1}^{8} p_{2} p_{4}+1172880 p_{1}^{7} p_{2}^{2} p_{3}+16562880 p_{1}^{6} p_{2} p_{6}+13554000 p_{1}^{6} p_{3} p_{5}+30841500 p_{1}^{5} p_{2}^{2} p_{5}+32993280 p_{1}^{4} p_{2}^{3} p_{4}+44292420 p_{1}^{4} p_{2}^{2} p_{3}^{2}+22402440 p_{1}^{3} p_{2}^{4} p_{3}+396886016 p_{1}^{4} p_{2} p_{8}+302522913 p_{1}^{4} p_{3} p_{7}+267245568 p_{1}^{4} p_{4} p_{6}+461080200 p_{1}^{3} p_{2}^{2} p_{7}+288584775 p_{1}^{3} p_{3}^{2} p_{5}+269170560 p_{1}^{3} p_{3} p_{4}^{2}+272937600 p_{1}^{2} p_{2}^{3} p_{6}+309372672 p_{1}^{2} p_{2}^{2} p_{4}^{2}+83637000 p_{1} p_{2}^{4} p_{5}+122778720 p_{1} p_{2}^{2} p_{3}^{3}+\frac{12}{7} p_{1}^{14}+3335307264 p_{1} p_{2} p_{3} p_{8}+2872105600 p_{1} p_{2} p_{4} p_{7}+2676888000 p_{1} p_{2} p_{5} p_{6}+2235382272 p_{1} p_{3} p_{4} p_{6}+1356192000 p_{2} p_{3} p_{4} p_{5}+51964416 p_{1}^{5} p_{2} p_{3} p_{4}+722094912 p_{1}^{3} p_{2} p_{3} p_{6}+651145600 p_{1}^{3} p_{2} p_{4} p_{5}+656734500 p_{1}^{2} p_{2}^{2} p_{3} p_{5}+550094976 p_{1}^{2} p_{2} p_{3}^{2} p_{4}+278069760 p_{1} p_{2}^{3} p_{3} p_{4}+1108344384 p_{1}^{2} p_{6}^{2}+421363712 p_{2}^{3} p_{8}+409239500 p_{2}^{2} p_{5}^{2}+210182144 p_{2} p_{4}^{3}+199624500 p_{3}^{3} p_{5}+284539392 p_{3}^{2} p_{4}^{2}+12367555200 p_{1} p_{13}+6772211712 p_{2} p_{12}+4930038036 p_{3} p_{11}+4067200000 p_{4} p_{10}+3614673600 p_{5} p_{9}+3384115200 p_{6} p_{8}+918 p_{1}^{11} p_{3}+4704 p_{1}^{10} p_{2}^{2}+66475 p_{1}^{9} p_{5}+191970 p_{1}^{8} p_{3}^{2}+533040 p_{1}^{6} p_{2}^{4}+3933041 p_{1}^{7} p_{7}+6482688 p_{1}^{6} p_{4}^{2}+7638273 p_{1}^{5} p_{3}^{3}+1720560 p_{1}^{2} p_{2}^{6}+138874500 p_{1}^{5} p_{9}+\frac{258762175}{2} p_{1}^{4} p_{5}^{2}+41211072 p_{1}^{2} p_{3}^{4}+10241280 p_{2}^{5} p_{4}+23788080 p_{2}^{4} p_{3}^{2}+2368405116 p_{1}^{3} p_{11}
                    \end{math}
				\end{minipage}
			\end{gathered}
		\end{equation}
        %\nopagebreak[4]
        \raggedbottom
         \begin{equation}
			\begin{gathered}
				Q_{15} = 
				\begin{minipage}[t]{1 \displaywidth}
					%\raggedright
					\linespread{1.2}
					\selectfont
					\begin{math}
                    163459296000 p_{15}+20415628800 p_{4} p_{5} p_{6}+16605275724 p_{1}^{2} p_{2}^{2} p_{9}+9518013918 p_{1}^{2} p_{3}^{2} p_{7}+7986713625 p_{1}^{2} p_{3} p_{5}^{2}+7488768000 p_{1}^{2} p_{4}^{2} p_{5}+6294863872 p_{1} p_{2}^{3} p_{8}+6055371000 p_{1} p_{2}^{2} p_{5}^{2}+3151167488 p_{1} p_{2} p_{4}^{3}+2989955250 p_{1} p_{3}^{3} p_{5}+4215780864 p_{1} p_{3}^{2} p_{4}^{2}+2828872512 p_{2}^{3} p_{3} p_{6}+2545728000 p_{2}^{3} p_{4} p_{5}+3402135000 p_{2}^{2} p_{3}^{2} p_{5}+3186012672 p_{2}^{2} p_{3} p_{4}^{2}+1886913792 p_{2} p_{3}^{3} p_{4}+100997922816 p_{1} p_{2} p_{12}+73498522284 p_{1} p_{3} p_{11}+60646400000 p_{1} p_{4} p_{10}+53899508025 p_{1} p_{5} p_{9}+50557353984 p_{1} p_{6} p_{8}+40816080000 p_{2} p_{3} p_{10}+34008864768 p_{2} p_{4} p_{9}+30628659200 p_{2} p_{5} p_{8}+29166752160 p_{2} p_{6} p_{7}+25543213056 p_{3} p_{4} p_{8}+23348722950 p_{3} p_{5} p_{7}+98385920 p_{1}^{5} p_{2}^{3} p_{4}+131607504 p_{1}^{5} p_{2}^{2} p_{3}^{2}+83577600 p_{1}^{4} p_{2}^{4} p_{3}+1186988032 p_{1}^{5} p_{2} p_{8}+906266466 p_{1}^{5} p_{3} p_{7}+796925952 p_{1}^{5} p_{4} p_{6}+1724128700 p_{1}^{4} p_{2}^{2} p_{7}+1077211575 p_{1}^{4} p_{3}^{2} p_{5}+1009958400 p_{1}^{4} p_{3} p_{4}^{2}+1359970560 p_{1}^{3} p_{2}^{3} p_{6}+1537148928 p_{1}^{3} p_{2}^{2} p_{4}^{2}+621208000 p_{1}^{2} p_{2}^{4} p_{5}+921050640 p_{1}^{2} p_{2}^{2} p_{3}^{3}+156702720 p_{1} p_{2}^{5} p_{4}+349742880 p_{1} p_{2}^{4} p_{3}^{2}+19702720000 p_{1}^{3} p_{2} p_{10}+14601393963 p_{1}^{3} p_{3} p_{9}+12327714816 p_{1}^{3} p_{4} p_{8}+11287058125 p_{1}^{3} p_{5} p_{7}+680960 p_{1}^{9} p_{2} p_{4}+2163240 p_{1}^{8} p_{2}^{2} p_{3}+35582976 p_{1}^{7} p_{2} p_{6}+28746000 p_{1}^{7} p_{3} p_{5}+76857000 p_{1}^{6} p_{2}^{2} p_{5}+\frac{422039538}{5} p_{3}^{5}+\frac{8}{5} p_{1}^{15}+\frac{9811577245}{3} p_{5}^{3}+20187936000 p_{1} p_{2} p_{3} p_{4} p_{5}+129217536 p_{1}^{6} p_{2} p_{3} p_{4}+2691009216 p_{1}^{4} p_{2} p_{3} p_{6}+2426860800 p_{1}^{4} p_{2} p_{4} p_{5}+3274254000 p_{1}^{3} p_{2}^{2} p_{3} p_{5}+2733087744 p_{1}^{3} p_{2} p_{3}^{2} p_{4}+2071296000 p_{1}^{2} p_{2}^{3} p_{3} p_{4}+24932450304 p_{1}^{2} p_{2} p_{3} p_{8}+21398803328 p_{1}^{2} p_{2} p_{4} p_{7}+19979784000 p_{1}^{2} p_{2} p_{5} p_{6}+16649349120 p_{1}^{2} p_{3} p_{4} p_{6}+14400795024 p_{1} p_{2}^{2} p_{3} p_{7}+12611137536 p_{1} p_{2}^{2} p_{4} p_{6}+11206180224 p_{1} p_{2} p_{3}^{2} p_{6}+1062 p_{1}^{12} p_{3}+5920 p_{1}^{11} p_{2}^{2}+102235 p_{1}^{10} p_{5}+311958 p_{1}^{9} p_{3}^{2}+1117120 p_{1}^{7} p_{2}^{4}+7400323 p_{1}^{8} p_{7}+13636608 p_{1}^{7} p_{4}^{2}+19128492 p_{1}^{6} p_{3}^{3}+8468992 p_{1}^{3} p_{2}^{6}+345532149 p_{1}^{6} p_{9}+384100260 p_{1}^{5} p_{5}^{2}+202786308 p_{1}^{3} p_{3}^{4}+17624448 p_{2}^{6} p_{3}+8841342708 p_{1}^{4} p_{11}+5497694208 p_{1}^{3} p_{6}^{2}+909144432 p_{2}^{4} p_{7}+92296713600 p_{1}^{2} p_{13}+24782355199 p_{1} p_{7}^{2}+27814506192 p_{2}^{2} p_{11}+15128277255 p_{3}^{2} p_{9}+11347796736 p_{3} p_{6}^{2}+10950460416 p_{4}^{2} p_{7}
                        \end{math}
                    \end{minipage}
			\end{gathered}
		\end{equation}
	\end{landscape}
	\clearpage
	
	\begin{landscape}
		\begin{table}[h]\label{Table:Origamis(g,n)}\caption{Number of complex origami of genus $g\leq 7$  and degree $d\leq 30$} \vspace*{0cm}
			\hspace*{0cm}	\begin{tabular}{|l|l|l|l|l|l|l|l|}
				\hline
				$d\backslash g$ & 1     & 2      & 3            & 4               & 5                   & 6                    & 7                         \\ \hline
				1     & 1     &        &              &                 &                     &                      &                           \\ \hline
				2     & 3/2   &        &              &                 &                     &                      &                           \\ \hline
				3     & 4/3   & 3      &              &                 &                     &                      &                           \\ \hline
				4     & 7/4   & 16     &              &                 &                     &                      &                           \\ \hline
				5     & 6/5   & 51     & 40           &                 &                     &                      &                           \\ \hline
				6     & 2     & 125    & 895/2        &                 &                     &                      &                           \\ \hline
				7     & 8/7   & 250    & 2645         & 1260            &                     &                      &                           \\ \hline
				8     & 15/8  & 480    & 10970        & 22716           &                     &                      &                           \\ \hline
				9     & 13/9  & 769    & 36210        & 204813          & 72576               &                      &                           \\ \hline
				10    & 9/5   & 1305   & 204455/2     & 1251364         & 1844992             &                      &                           \\ \hline
				11    & 12/11 & 1845   & 252963       & 5897515         & 22898872            & 6652800              &                           \\ \hline
				12    & 7/3   & 2905   & 1153133/2    & 68997247/3      & 755203519/4         & 221042352            &                           \\ \hline
				13    & 14/13 & 3773   & 1196407      & 77420987        & 1177955068          & 3541805088           & 889574400                 \\ \hline
				14    & 12/7  & 5622   & 2376781      & 232385498       & 11951703583/2       & 37242034072          & 36804223224               \\ \hline
				15    & 8/5   & 6982   & 4375513      & 632608951       & 25793853778         & 1463867417063/5      & 2187354789568/3           \\ \hline
				16    & 31/16 & 9920   & 7893228      & 1595092400      & 97761236930         & 1850410329588        & 9404060075112             \\ \hline
				17    & 18/17 & 11700  & 13323270     & 3745811960      & 332544854429        & 9847878739864        & 89951590073796            \\ \hline
				18    & 13/6  & 16523  & 44766941/2   & 24999351608/3   & 1033399238519       & 45566512617381       & 4117963046068735/6        \\ \hline
				19    & 20/19 & 18615  & 35362698     & 17543425750     & 2969117673405       & 187597330815919      & 4376648043083560          \\ \hline
				20    & 21/10 & 25389  & 112412637/2  & 35504174989     & 31918356295471/4    & 3497439457170376/5   & 24092889582888032         \\ \hline
				21    & 32/21 & 28720  & 84383518     & 206113869715/3  & 20197837305287      & 351691666938395144/3 & 15873381482595046147/7    \\ \hline
				22    & 18/11 & 37755  & 256592113/2  & 129103300974    & 97127880876161/2    & 7607528525353639     & 1026717048398989327/2     \\ \hline
				23    & 24/23 & 41118  & 184617019    & 233615759144    & 111343300254203     & 22632300940847170    & 2051782644293906901       \\ \hline
				24    & 5/2   & 55345  & 542403593/2  & 413542969895    & 980386498607549/4   & 63524925757973979    & 45419662418857625419/6    \\ \hline
				25    & 31/25 & 58001  & 376719701    & 708449494342    & 518957184833896     & 845989628823907901/5 & 26018590762599856469      \\ \hline
				26    & 21/13 & 75495  & 1073541961/2 & 1193844359250   & 2126383631890665/2  & 430034339069982127   & 167910837473121614903/2   \\ \hline
				27    & 40/27 & 81430  & 725727194    & 5863055947867/3 & 2108534616523198    & 1047135281059367535  & 255943686574745486898     \\ \hline
				28    & 2     & 102942 & 1007386665   & 3160184047614   & 8139312351181037/2  & 2453325851086614562  & 741262551284963908481     \\ \hline
				29    & 30/29 & 106785 & 1327427217   & 4977886895985   & 7640320671017787    & 5546005695691838947  & 2048919731443074360567    \\ \hline
				30    & 12/5  & 139242 & 1807688043   & 7772645020410   & 28044834986384939/2 & 12138674333408224675 & 32565071476679640061757/6 \\ \hline
			\end{tabular}
		\end{table}
	\end{landscape}
	
	\begin{landscape}
		\begin{table}[h]\caption{Number of complex origami of genus 
				$g=8\ldots 10$ and degree $d\leq 30$} \vspace*{0cm}
			\hspace*{0cm}	\begin{tabular}{|l|l|l|l|}
				\hline
				$d\backslash g$ & 8                          & 9                            & 10                              \\ \hline
				
				\iffalse
				
				\fi 15    & 163459296000               &                              &                                 \\ \hline
				16    & 8143111914624              &                              &                                 \\ \hline
				17    & 193384506243840            & 39520825344000               &                                 \\ \hline
				18    & 2973952153937904           & 2312882651411712             &                                 \\ \hline
				19    & 33724882414394000          & 64337950293173568            & 12164510040883200               \\ \hline
				20    & 303335630797746596         & 1154562366493598912          & 820634271502694400              \\ \hline
				21    & 15214878006528625536       & 26262986467481393472         & 4644631106519040000             \\ \hline
				22    & 14555613814464572932       & 158352633319758487184        & 540735568245302924352           \\ \hline
				23    & 82163396151489103668       & 1364004293197120038484       & 8150321022900364573056          \\ \hline
				24    & 415373154121461526161      & 80372287483019487329407/8    & 96705245220413325506992         \\ \hline
				25    & 1907804084050076167468     & 64813171985485528190432      & 946502130038492766756112        \\ \hline
				26    & 8053321768082407917400     & 373030655339582803158424     & 7895472549944594518993908       \\ \hline
				27    & 31539903289944316470903    & 1943368628582588135753249    & 57502313484416742556875831      \\ \hline
				28    & 808548802975291149510846/7 & 37088009364066173383333343/4 & 372451801484855535477578996     \\ \hline
				29    & 398189251690176385353627   & 40904453684459735325696824   & 2177071023099342103410800332    \\ \hline
				30    & 1299460224149740821910012  & 168200816668468081373046107  & 34860944808099359359034732131/3 \\ \hline
			\end{tabular}
			%\end{table}
			%\end{landscape}
			%\nopagebreak
			%\begin{landscape}
			%\begin{table}[h]
			\vspace*{0.5cm}            
			\caption{Number of complex origami of genus $g=11...13$ and degree $d\leq 30$}\vspace*{0cm}
			\hspace*{0cm}	\begin{tabular}{|l|l|l|l|}
				\hline
				$d\backslash g$ & 11                               & 12                              & 13                                  \\ \hline
				\iffalse	
				\fi 22    & 355818427261727293440            &                                 &                                     \\ \hline
				23    & 12914363474580728203776          & 2154334728240414720000          &                                     \\ \hline
				24    & 300943344434082542732160         & 185155053842445131980800        &                                     \\ \hline
				25    & 5121654150794434759618880        & 7532524379814235890656256       & 1193170003333152768000000           \\ \hline
				26    & 68439945654406252402152384       & 196450713160102079377793280     & 113894735112136044221644800         \\ \hline
				27    & 752423319177440922554720240      & 3734862321209227772164887552    & 5143210219224978114733350912        \\ \hline
				28    & 7031567014599687661186104688     & 55639757000650670155352449152   & 148721581521973424183863048128      \\ \hline
				29    & 57221172732267012922115463604    & 680496602208863462334037722368  & 3130300901582416277061327970560     \\ \hline
				30    & 826142276193803767564436150163/2 & 7059375340054134070837385574896 & 154630904804283210742240067755264/3 \\ \hline
			\end{tabular}
		\end{table}
	\end{landscape}
	
	\begin{table}[h]\caption{Number of complex origami of genus $g=14,15$  and degree $d\leq 30$} \hspace*{0cm}
		\begin{tabular}{|l|l|l|}
			\hline
			$d\backslash g$ & 14                                & 15                               \\ \hline
			\iffalse			
			\fi 27    & 777776389315596582912000000       &                                  \\ \hline
			28    & 81764157739216967337428582400     &                                  \\ \hline
			29    & 4064822526732317856433918771200   & 589450799582646796969574400000   \\ \hline
			30    & 129285359717415405598641857329152 & 67753528732337683444789063680000 \\ \hline
		\end{tabular}
	\end{table}

\end{document}